\documentclass[]{interact}

\usepackage{xcolor}
\usepackage{epstopdf}
\usepackage[caption=false]{subfig}

\usepackage[numbers,sort&compress]{natbib}
\bibpunct[, ]{[}{]}{,}{n}{,}{,}

\theoremstyle{plain}
\newtheorem{theorem}{Theorem}[section]
\newtheorem{lemma}[theorem]{Lemma}
\newtheorem{corollary}[theorem]{Corollary}
\newtheorem{proposition}[theorem]{Proposition}

\theoremstyle{definition}

\newtheorem{assumption}[theorem]{Assumption}

\theoremstyle{remark}

\def\be{\begin{eqnarray}}
\def\ee{\end{eqnarray}}

\def\b*{\begin{eqnarray*}}
\def\e*{\end{eqnarray*}}

\def \E{\mathbb{E}}
\def \F{\mathbb{F}}

\def \H{\mathbb{H}}

\def \L{\mathbb{L}}

\def \P{\mathbb{P}}

\def \R{\mathbb{R}}

\def\Ac{{\cal A}}

\def\Fc{{\cal F}}

\def\Mc{{\cal M}}

\def\Pc{{\cal P}}

\def\1{{\bf 1}}

\newcommand{\dd}{\mathop{}\mathopen{}\mathrm{d}}
\def\Dpin{\Delta^{\!\pi_n}\!}
\begin{document}


\title{It\=o and It\=o\,-Wentzell chain rule for flows of conditional laws 
\\ of continuous semimartingales: an easy approach\thanks{The authors declare that this research presents no conflict of interest.}}

\author{
\name{Assil Fadle\textsuperscript{a}\thanks{assil.fadle@polytechnique.edu.}, Mehdi Talbi \textsuperscript{b}\thanks{Université Paris-Cité, Paris, mtalbi@lpsm.paris} and Nizar Touzi\textsuperscript{c}\thanks{NYU, Tandon School of Engineering, USA. This author is partially supported by NSF grant \#DMS-2508581,  nizar.touzi@nyu.edu}}
\affil{\textsuperscript{a}Ecole Polytechnique, France; \\ \textsuperscript{b}Université Paris-Cité, France; \\ \textsuperscript{c}New York University, Tandon School of Engineering, USA.}
}

\maketitle

\begin{abstract}
We provide a general It\=o\,-Wentzell formula for a random field of maps on the Wasserstein space of probability measures, defined by continuous semimartingales, and evaluated along the flow of conditional distributions of another continuous semimartingale. Our method follows standard arguments of It\=o calculus, and thus bypasses the approximation by empirical measures commonly used in the existing literature. As an application, we derive the dynamic programming equation for a mean field stochastic control problem with common noise.
\end{abstract}

\begin{amscode}
60G40, 49N80, 35Q89, 60H30
\end{amscode}


\section{Introduction}

Let $X=(X_t)_{t\ge 0}$ be a square integrable continuous semimartingale on a filtered probability space $\big(\Omega,\Fc,\{\Fc_t\}_{t\ge 0},\P\big)$. For simplicity, we consider the scalar case as the multidimensional extension does not raise any special difficulties.  Denote by $m_t:=\P\circ X_t^{-1}$ the marginal law of $X_t$, which lies in the set $\Pc_2(\R)$ of all probability measures with finite second moment. For a function $u:\Pc_2(\R)\longrightarrow\R$, with appropriate regularity, an It\=o's chain rule for the map $t\longmapsto u(m_t)$ was established by various methods in the literature after the Lectures of P.L. Lions at the Coll\`ege de France, see Buckdahn, Li, Peng \& Rainer \cite{buckdahn2017} and Chassagneux, Crisan \& Delarue \cite{chassagneux2015} for continuous diffusions, Cavallazzi \cite{cavallazzi2022} for a Krylov-type extension of the It\=o formula to maps in appropriate Sobolev spaces, Li \cite{li2} and Burzoni, Ignazio, Reppen \& Soner \cite{burzoni2020} for special classes of jump-diffusions with continuous marginals. The case of general c\`adl\`ag semimartingales was solved simultaneously by Guo, Pham \& Wei \cite{guo2022} and Talbi, Touzi \& Zhang \cite{talbi2023}.

The It\=o chain rule states that, for a map $u$ with appropriate smoothness on the Wasserstein space of probability measures, we have
\begin{equation}\label{Ito0}
u(m_t)
=
u(m_0)
+\E\Big[\int_0^t \!\!\partial_\mu u(m_s,X_s)\dd X_s
               \!+\!\frac12\int_0^t\!\! \partial_x\partial_\mu u(m_s,X_s)\dd\langle X\rangle_s
     \Big],
~\mbox{for all}~
t\ge 0,
\end{equation}
where $\partial_\mu u:\Pc_2(\R)\times\R\longrightarrow\R$ denotes the so-called Lions derivative, and $\partial_x$ is the partial gradient operator with respect to the $x-$variable. See e.g. Carmona \& Delarue \cite{CD1}.

Our objective in this paper is to revisit extensions of the last It\=o's rule in two directions:

\noindent - the measure variable is random and defined as the conditional law $\mu_t:=\P\circ (X_t|\Fc^0)^{-1}$ of $X_t$ given some sub-sigma algebra $\Fc^0$ of $\Fc$,

\noindent - the function $u$ is extended to the context of a dynamic stochastic flow of continuous semimartingales $\{u_t(x),t\ge 0\}$ for all fixed $x\in\R$. 

The first extension is motivated by the vibrant research activity on mean field stochastic control with common noise, and the Master equation in the context of mean field games with common noise. The second extension is also motivated by similar stochastic control problems under partial information. The huge interest of the community in this area is enhanced by the wide applications in various questions pertaining to multiple agents decision problems.

Our main emphasis is on the simplicity of our derivations which follow standard arguments in It\=o calculus, and which allow to obtain new extensions which were not considered in the existing literature. In order to better explain our approach, let us show how \eqref{Ito0} can be obtained by means of the following early graduate class level arguments (where the two first steps are simple reminders):
\begin{itemize}
\item We first recall from Cardaliaguet, Delarue, Lasry \& Lions \cite{cardaliaguet2015} that the Lions derivative $\partial_\mu$ is related to the functional linear derivative $\delta_m$ by $\partial_\mu=\partial_x\delta_m$, where $\delta_mu:\Pc_2(\R)\times\R\longrightarrow\R$ is defined for all $m^0,m^1\in\Pc_2(\R)$, with barycenter $m^\lambda:=(1-\lambda)m^0+\lambda(m^1\!\!-\!m^0)$, by the following limit, if exists: 
$$
\lim_{\lambda\searrow 0}
\frac1\lambda\big[u\big(m^\lambda\big)-u(m^0)\big]
=:
\langle\delta_m u(m^0),m^1\!\!-\!m^0\rangle
=
\int\delta_m u(m^0,x)(m^1\!\!-\!m^0)(\dd x).
$$
Notice that this definition is a mix of directional and G\^ateaux derivative, and that the map $\delta_mu(m^0):=\delta_mu(m^0,.):\R \longrightarrow\R$ needs to have quadratic growth, at most, in order for the last integral to be well-defined.
\item By standard calculus, and under slight regularity, this definition is equivalent to the requirement of existence of such a function $\delta_mu$ satisfying the requirement
$$
u(m^1)-u(m^0)
=
\int_0^1 \big\langle \delta_m u\big(m^\lambda\big),m^1\!\!-\!m^0\big\rangle \dd\lambda,
~\mbox{for all}~
m^0,m^1\in\Pc_2(\R),
$$
which is all we need for our subsequent derivation of It\=o's formula.
\item Given a dense partition $(t^n_i)_{i\ge 0}$ of the $[0,\infty)$, denote $s^n_i:=s\wedge t^n_i$, for all $s\ge 0$, and use the telescopic decomposition together with the last definition to see that:
\be\label{telescopic0}
u(m_s)-u(m_0)
=
\sum_{i\ge 1} u(m_{s^n_i})-u(m_{s^n_{i-1}})
=
\sum_{i\ge 1} \int_0^1 
\langle U_{n,i}^\lambda,m_{s^n_i}-m_{s^n_{i-1}}
 \rangle\dd\lambda,
\ee 
with $U_{n,i}^\lambda:=\delta_m u\big((1-\lambda)m_{s^n_{i-1}}
                                  +\lambda m_{s^n_i}\big)$ a scalar map on $\R$.
We next observe that 
\b*
\langle U_{n,i}^\lambda,m_{s^n_i}\!-\!m_{s^n_{i-1}}\rangle
&=&
\int U_{n,i}^\lambda \dd(m_{s^n_i}\!-\!m_{s^n_{i-1}})
\\
&=&
\E\big[U_{n,i}^\lambda(X_{s^n_i})\!-\!U_{n,i}^\lambda(X_{s^n_{i-1}})\big]
\\
&=&
\E\Big[\int_{s_{i-1}^n}^{s_i^n}\!\!\!(U_{n,i}^\lambda)'(X_r)\dd X_r
\!+\!\frac12(U_{n,i}^\lambda)''(X_r)\dd\langle X\rangle_r
\Big],
\e* 
where the last equality follows from the standard It\=o's formula, under the appropriate regularity assumptions on the map $U_{n,i}^\lambda$. Plugging this expression in \eqref{telescopic0}, we obtain the required formula \eqref{Ito0} by standard limiting argument using the dominated convergence theorem.
\end{itemize}

The last argument is most appealing as it uses the standard intuitive notion of functional linear derivative $\delta_m$. Moreover, it completely bypasses the crucial step of projection on empirical measures used in most of the previous literature following the Lectures of P.L. Lions at the Coll\`ege de France, see Chassagneux, Crisan \& Delarue \cite{chassagneux2015}, Buckdahn, Li, Peng \& Rainer \cite{buckdahn2017}, Carmona \& Delarue \cite{CD1}. Here, the idea is to approximate the marginal law $m_t$ by the corresponding empirical measure $\overline{m}_t^N:=\frac1N\sum_{i\le N} \delta_{X^i_t}$ of a finite sample of $N$ independent copies $(X^1,\ldots,X^N)$, apply the standard It\=o's formula to the finite dimensional map $\overline{u}^N(X^1_t,\ldots,X^N_t):=u(\overline{m}_t^N)$, and finally take limits by using fine results on the convergence of empirical measures.

The simple method outlined above is applied in Talbi, Touzi \& Zhang \cite{talbi2023} in the context of c\`ad-l\`ag semimartingales, see also the parallel paper by Guo, Pham \& Wei \cite{guo2022} which uses a functional analytic extension of an appropriate class of cylindrical maps in order to account for the jumps of the semimartingale. 

The main contribution of this paper is to show that the previous simple method also applies to derive an It\=o\,-Wentzell chain rule for conditional laws. This answers in particular a question raised in dos Reis \& Platonov \cite{reis2022}, who derive the It\=o\,-Wentzell formula by adapting the technique of projection on empirical measures used by Carmona \& Delarue \cite{CD2} to derive the It\=o formula for conditional laws, see also Cardaliaguet, Delarue, Lasry \& Lions \cite{cardaliaguet2015} in the context of the Master equation. We notice that, while the state process in \cite{reis2022} is defined by SDEs driven by Brownian motions and is conditioned by a Brownian motion, we consider in this paper general continuous semimartingales with general conditioning. Moreover, our It\=o\,-Wentzell formula is derived for a random flow continuous semimartingale, extending the case of deterministic function of a process and a conditional law of another process of \cite{reis2022}.

The paper is organized as follows. Section \ref{sec:Ito} provides an It\=o's formula for conditional marginal laws of continuous semimartingales. Although this result is a particular case of the subsequent one, we believe that it deserves to be isolated for the sake of clarity. Section \ref{sec:Ito-Wentzell} contains our general It\=o\,-Wentzell formula in the context where the random field of maps is also defined by continuous semimartingales. Finally, Section \ref{sec:control} provides an application in mean field stochastic control with common noise.

\section{Notations}
We denote $x\cdot y:=\sum_i x_iy_i$ the Euclidean scalar product of two vectors in any finite dimensional space, $A\!:\!B:=\mbox{Tr}[A^\intercal B]$ and $A^{\otimes 2} = A A^\intercal$ for all matrices of appropriate size.

Throughout this paper, we fix a constant maturity $T>0$, and a complete probability space $(\Omega,\Fc,\P)$ equipped with a filtration $\F=\{\Fc_t,0\le t\le T\}$. 

Let $\pi : 0 = t_0 < \ldots < t_{n_\pi}=T$ be a subdivision of $[0,T]$ with mesh size $|\pi|:=\max_{i\le n_\pi} (t_i-t_{i-1})$. A sequence of subdivisions $(\pi^n)_{n \geq 0}$ is dense if the sequence of meshes $| \pi^n |$ converges to $0$ as $n \rightarrow \infty$. 

A stochastic process is said to be piecewise constant along the subdivision $\pi$ if it is constant on each interval $(t_{i-1},t_i]$. For a process $Y$ valued in $\R^d$, we denote the increment to the subdivision by $\Delta^{\pi} Y_s = Y_s - Y_{t_{i_s-1}}$ if $s\in(t_{i_s-1},t_{i_s}]$, and if the process $Y$ is in addition piecewise constant along $\pi$, we have $\Delta^{\pi} Y_s=Y_{t_{i_s}} - Y_{t_{i_s-1}}$.

The total variation of $Y$ is denoted by 
	\b*
	|Y|_{\rm TV} 
	&=& 
	\sup_\pi \sum_{i=1}^{n_\pi-1} |Y_{t_{i+1}} - Y_{t_i}| 
	\;=\; 
	\sup_\pi \sum_{i=1}^{n_\pi-1} |\Delta^\pi Y_{t_{i+1}}|,
	\e*
where $|.|$ is the Euclidean norm in $\R^d$.

The quadratic variation of $Y$ is defined as
\begin{align*}
	\langle Y \rangle_s = \lim_{|\pi| \rightarrow 0} \sum_{i=1}^{n_\pi - 1} (Y_{t_{i+1}} - Y_{t_i})(Y_{t_{i+1}} - Y_{t_i})^\intercal, \quad
	\mbox{for all} \quad s\ge 0,
\end{align*}
where the limit is in probability and does not depend on the choice of the subdivisions sequence.

$X$ is said to be a (continuous) semimartingale if it can be written as $X_s = X_0 + A_s + M_s, s \in [0, T]$ where $A$ is a (continuous) finite-variation process and $M$ is a (continuous) martingale.

 $\H^2(Y)$ is the collection of all progressively measurable processes $H$, with same dimension as $Y$, such that $\E\big[\int_0^T H_sH_s^\intercal\!:\! d\langle Y\rangle_s\big]<\infty$.

 A sequence $(H^n)_{n\ge 0}$ of predictable bounded processes is called a simple approximation of a process $H\in\H^2(Y)$ if there exists a dense sequence of subdivisions $(\pi^n)_{n \geq 0}$ such that $H^n$ is piecewise constant along $\pi^n$, for all $n\ge 0$, and $H^n\longrightarrow H$ in $\mathbb{H}^2(Y)$, as $n\to\infty$.

The following (probably well-known) result will be used frequently. As we failed to find a reference for it, we report its proof as a complement in the Appendix section \ref{sec:App}.

	\begin{lemma} \label{lemma}
Let $X$ be a semimartingale with decomposition $X=A+M$ into a finite variation process $A$ and a martingale $M$ satisfying $\mathbb{E}\left[ |A|_{\rm{TV}}^2 + \langle M \rangle_T^2 \right] < \infty$. Let $(H^n)_{n\ge 0}$ be a simple approximation of a matrix-valued bounded progressively measurable $H$ with rows in $\mathbb{H}^2(X)$, along some sequence of subdivisions $(\pi^n : 0 = t_0^n < \ldots < t_{p_n}^n = T)$. Then:
\b*
\sum_{i=1}^{p_n}  H^n_{t_{i-1}^n} \! : \! (\Delta^{\pi^n} 
                             X_{t_i^n}) (\Delta^{\pi^n} X_{t_i^n})^\intercal 
\longrightarrow 
\int_0^T H_s  \! : \! \dd \langle X \rangle_s,
&\mbox{in}~~\L^1&
\mbox{as}~~n\to\infty.
\e*
\end{lemma}
	
We notice that, if $H$ is continuous, then the last convergence result holds true with $H^n_{t_i^n} = H_{t_i^n}$, for $i=0,\ldots,n_{p_n}$.
	
The marginal laws considered in this paper lie in the set $\Pc_2(\R^d)$ of all probability measures on $\R^d$ with finite second moment. Similarly, our conditional marginal laws are random maps taking values in $\Pc_2(\R^d)$. This set is naturally equipped with the Wasserstein distance
\b*
\mathcal{W}_2(m,m')
:=
\inf_{\pi\in\Pi(m,m')} \int |x-x|^2\pi(\dd x,\dd x'),
&\mbox{for all}&
m,m'\in\Pc(\R^d),
\e*
where $\Pi(m,m')$ is the set of all couplings of $(m,m')$, i.e. probability measures on $\R^d\times\R^d$ with marginals $m$ and $m'$.

We say that a function $u : \mathcal{P}_2(\mathbb{R}^d) \rightarrow \mathbb{R}$ admits a (first order) functional linear derivative if there exists a map $\delta_m u : \mathcal{P}_2(\mathbb{R}^d) \times \mathbb{R}^d \rightarrow \mathbb{R}$ such that for all $m^0, m^1 \in \mathcal{P}_2(\mathbb{R}^d)$,
	\begin{equation*}
		u(m^1) - u(m^0) = \int_0^1 \int_{\mathbb{R}^d} \delta_m u(m^\lambda, x) (m^1 - m^0)(\dd x) \dd \lambda \text{ with } m^\lambda := (1-\lambda)m^0+\lambda m^1,
	\end{equation*}
	and $\delta_m u$ has quadratic growth in $x$, locally uniformly in $m$, so that the last integral is well-defined. Similarly, the second order functional linear derivative $\delta_m^2 : \mathcal{P}_2(\mathbb{R}^d) \times \mathbb{R}^d \times \mathbb{R}^d \rightarrow \mathbb{R}$ is such that for all $m, m' \in \mathcal{P}_2(\mathbb{R}^d)$, and $x \in \mathbb{R}^d$:
	\begin{equation*}
		\delta_m u(m^1,x) - \delta_m u(m^0, x) 
		= 
		\int_0^1 \int_{\mathbb{R}^d} \delta_m^2 u(m^\lambda, x, \hat{x}) (m^1 - m^0)(\dd \hat{x}) \dd \lambda,
	\end{equation*}
	and $\delta_m^2 u$ has quadratic growth in $\hat{x}$, locally uniformly in $m$, for all fixed $x\in\R^d$. Notice that under these conditions, $\delta_m \partial_x \delta_m u = \partial_x \delta_m^2 u$.

\section{It\=o's formula}
\label{sec:Ito}
	
Throughout this paper, we consider an $\R^d-$valued continuous semimartingale with canonical decomposition 
	\begin{equation}\label{def:X}
		X =X_0 + M + A,
	\end{equation}
where $M$ is a martingale and $A$ is a finite-variation process, both started from $0$, and $X_0\in\L^2(\Fc_0)$. We assume that:
\begin{equation}\label{eq:integrability}
\mathbb{E}\Big[ \sup_{t \in [0,T]} | X_t |^2 \Big] < \infty.
\end{equation}

For an arbitrary sub-filtration $\mathbb{F}^0 = (\mathcal{F}^0_t)_{t \ge 0}$ of $\mathbb{F}$, we denote by $\mu_t = \mathcal{L}(X_t | \mathcal{F}^0_T)$, the law of $X_t$ conditional on $\mathcal{F}^0_T$, for $t \in [0, T]$.  Note that, for all $s, t \in [0,T]$,
$$ \mathcal{W}_2(\mu_s, \mu_t)^2 \le \mathbb{E}\Big[ \big| X_s  - X_t \big|^2 \big| \mathcal{F}_T^0 \Big] \underset{s \to t}{\longrightarrow} 0, \ \mbox{a.s.,}$$
by the integrability condition \eqref{eq:integrability}. Therefore, we may consider an almost surely continuous version of the flow $t \mapsto \mu_t$.

Our first result is the following It\=o's formula for flows of conditional law of the continuous semimartingale $X$.

\begin{assumption}\label{assum:Ito}
The map $u\!: \mathcal{P}_2(\mathbb{R}^d) \longrightarrow \mathbb{R}$, and the continuous semimartingale $X$ satisfy:

 {\rm(I1)} $\delta_m u, \partial_x^2 \delta_m u, \delta_m^2 u, \partial_x \partial_{\hat{x}} \delta_m^2 u$ exist and are continuous in each variable;

 {\rm(I2)} $\partial_x^2 \delta_m u, \partial_x \partial_{\hat{x}} \delta_m^2 u$ are bounded;

 {\rm(I3)} $X_0$, $|A|_{\rm{TV}}$ and $\langle M \rangle_T$ are square integrable.
\end{assumption}

\begin{theorem}\label{thm:Ito}
Let Assumption \ref{assum:Ito} hold true, and let $\hat{X}$ be a copy of $X$ on a copy probability space, with $\mathcal{L}(\hat X | \mathcal{F}^0_T) = \mathcal{L}(X | \mathcal{F}^0_T)$. Then:
\b*
u(\mu_T) - u(\mu_0) 
&=& 
\mathbb{E}^0 \left[ \int_0^T  \partial_x \delta_m u(\mu_s, X_s) \!\cdot\!\dd X_s 
                              + \frac12\;\partial^2_x \delta_m u(\mu_s, X_s) \!:\!\dd \langle X \rangle_s \right]  \\
&&
+ \; \mathbb{E}^0 \hat{\mathbb{E}}^0 \left[  \int_0^T\frac12\; \partial^2_{x \hat{x}} \delta_m^2 u (\mu_s, X_s, \hat{X}_s) \!:\!\dd \langle X, \hat{X} \rangle_s \right], 
~~\mbox{a.s.}
\e*
where  $\mathbb{E}^0 := \mathbb{E} [ \cdot | \mathcal{F}^0_T]$ and $\hat{\mathbb{E}}^0:= \mathbb{E} [ \cdot | X,\mathcal{F}^0_T]$ denote the conditional expectations in the enlarged space.
\end{theorem}
	
\begin{proof} We organize our arguments in three steps.

\noindent {\it Step 1.} Let $\pi^n : 0 = t_0^n < t_1^n < \ldots < t_{p_n}^n = T$ be a dense sequence of partitions of $[0,T]$, and denote $\mu_{t_i^n}^\lambda := \lambda \mu_{t_i^n} + (1-\lambda) \mu_{t_{i-1}^n}$. By the definition of the linear functional derivative, we have:
\be
\delta_i^nu
:=
u(\mu_{t_i^n}) - u(\mu_{t_{i-1}^n}) 
&\!\!\!\!=&\!\!\!\! 
\int_0^1 \int_{\mathbb{R}^d} 
\delta_m u(\mu_{t_i^n}^\lambda, x) (\mu_{t_i^n} - \mu_{t_{i-1}^n})(\dd x) 
\dd \lambda 
\nonumber\\
&\!\!\!\!=&\!\!\!\! 
\int_0^1 \mathbb{E}^0 \left[ \delta_m u(\mu_{t_i^n}^\lambda, X_{t_i^n}) 
                                            - \delta_m u(\mu_{t_i^n}^\lambda, X_{t_{i-1}^n}) 
                                    \right] \dd \lambda.~~~~~~
\label{decompo0}
\ee

By the second order Taylor theorem, we may rewrite this as
$$
\delta_i^n u
\!=\!\! \int_0^1\!\!\! \mathbb{E}^0 \!\left[ \partial_x \delta_m u (\mu_{t_{i}^n}^\lambda, X_{t_{i-1}^n}) \! \cdot \!\Dpin X_{t_i^n} 
\!+\! \dfrac{1}{2} \partial^2_x \delta_m u  (\mu_{t_{i}^n}^\lambda, \xi_{t_{i-1}^n}) \! : \!(\Dpin X_{t_i^n}) (\Dpin X_{t_i^n})^\intercal \right] \dd \lambda,
$$
for some r.v. $\xi_{t_{i-1}^n}$ lying between $X_{t_{i-1}^n}$ and $X_{t_i^n}$. Let us introduce an independent copy $\hat{X}$ of $X$ conditionally to $\mathbb{F}^0$. Using the notation $
\left\lbrace F(\theta, \cdot) \right\rbrace_{\theta'_1}^{\theta'_2} := F(\theta, \theta'_2) - F(\theta, \theta'_1), 
$
for all map $F(\theta,\theta')$, and denoting $\gamma :=\partial_x \delta_m u$, we compute that
\b*
\Big\{\gamma(., X_{t_{i-1}^n})
\Big\}_{_{\mu_{t_{i-1}^n}}}^{^{\mu_{t_{i-1}^n}^\lambda}}
&=&
\int_0^1\!\!\!\int\!\! \delta_m \gamma (\mu_{t_{i}^n}^{\lambda \lambda'}, X_{t_{i-1}^n}, \hat{x}) (\mu_{t_{i}^n}^\lambda \!-\! \mu_{t_{i-1}^n})(\dd \hat{x}) \dd \lambda' 
\\
&=& 
\lambda \int_0^1 \hat{\mathbb{E}}^0 \left[ \left\lbrace \delta_m \gamma(\mu_{t_{i}^n}^{\lambda \lambda'}, X_{t_{i-1}^n}, \cdot) \right\rbrace_{\hat{X}_{t_{i-1}^n}}^{\hat{X}_{t_i^n}}   \right] \dd \lambda' \\
&=&
\lambda \int_0^1 \hat{\mathbb{E}}^0 \left[ \partial_{\hat{x}} \delta_m \gamma(\mu_{t_{i}^n}^{\lambda \lambda'}, X_{t_{i-1}^n}, \hat{\xi}_{t_{i-1}^n}) \Dpin \hat{X}_{t_i^n} \right] \dd \lambda',
\e*
for some $\hat{\xi}_{t_{i-1}^n}$ between $\hat{X}_{t_{i-1}^n}$ and $\hat{X}_{t_i^n}$.  This provides:
\b*
\delta_i^n u 
&=& 
\mathbb{E}^0 \left[ \partial_x \delta_m u(\mu_{t_{i-1}^n}, X_{t_{i-1}^n}) \! \cdot \! \Dpin X_{t_i^n} \right]\\
\\
&&
+\; \mathbb{E}^0 \hat{\mathbb{E}}^0 \left[ \int_0^1\!\!\! \int_0^1 \lambda \partial_{\hat{x}} \delta_m \partial_x \delta_m u (\mu_{t_{i}^n}^{\lambda \lambda'}, X_{t_{i-1}^n}, \hat{\xi}_{t_{i-1}^n}) \! : \! (\Dpin X_{t_i^n}) (\Dpin \hat{X}_{t_i^n})^\intercal \dd \lambda' \dd \lambda  \right] 
\\
&&+ \int_0^1 \mathbb{E}^0 \left[ \dfrac{1}{2} \partial^2_x \delta_m u  (\mu_{t_{i}^n}^\lambda, \xi_{t_{i-1}^n}) \! : \! (\Dpin X_{t_i^n})(\Dpin X_{t_i^n})^\intercal  \right] \dd \lambda.
\e*
Summing \eqref{decompo0}, and denoting by $t^n(s)$ the closest subdivision point strictly to the left of $s$, this provides:
\b*
u(\mu_T) - u(\mu_0)
&=& 
\int_{0}^{T} \mathbb{E}^0 \left[ \partial_x \delta_m u(\mu_{t^n(s)}, X_{t^n(s)}) \! \cdot \! \dd X_s \right] \\
&& \hspace{-27mm}
+\sum_{i=1}^{p_n} \int_0^1 \mathbb{E}^0\!\! \left[ \dfrac{1}{2} \partial^2_x \delta_m u(\mu_{t_{i}^n}^\lambda, \xi_{t_{i-1}^n}) \! : \! (\Dpin X_{t_i^n} )(\Dpin X_{t_i^n} )^\intercal  \right] \dd \lambda \\
&& \hspace{-27mm}
+ \sum_{i=1}^{p_n} \mathbb{E}^0 \hat{\mathbb{E}}^0\!\!
\left[ \int_0^1\!\!\!\int_0^1 \lambda \partial_{\hat{x}} \delta_m \partial_x \delta_m u (\mu_{t_{i}^n}^{\lambda \lambda'}, X_{t_{i-1}^n}, \hat{\xi}_{t_{i-1}^n}) \! : \! (\Dpin X_{t_i^n} ) (\Dpin \hat{X}_{t_i^n} )^\intercal \dd \lambda' \dd \lambda  \right].
\e*
Our goal in the subsequent steps is to analyse the convergence of each term in the last decomposition, along a suitable sequence of subdivisions. 

\noindent {\it Step 2.} In this step, we start with the first term that we denote $U^n_1$. Namely,
\b*
U^n_1
&:=&
\sum_{i=1}^{p_n}  \mathbb{E}^0\left[ \int_{t_{i-1}^n}^{t_i^n} 
\partial_x \delta_m u(\mu_{t_{i-1}^n}\!, \!X_s)\! \cdot \!\dd X_s \right] 
\\
&=&  
\mathbb{E}^0 \left[ \int_0^T f_n(s) \! \cdot \! \dd X_s \right] 
~~\mbox{with}~~ 
f_n(s) := \partial_x \delta_m u(\mu_{t^n(s)}, \!X_s).
\e*
Since $\mu_s$ is continuous in $s$, $f_n(s) \longrightarrow f(s) = \partial_x \delta_m u(\mu_s, X_s)$ almost surely when $n \rightarrow + \infty$. Moreover, $|\partial_x \delta_m u(\mu_s, X_s)| \leq C (1 + |X_s|)$, by our assumption on $\partial_x^2 \delta_m u$, and as $\mathbb{E}[\langle M \rangle_T + |A|_{\text{TV}}^2] < \infty$, we also have $\mathbb{E} \left[ \sup_{0 \leq t \leq T} |X_t|^2 \right] < \infty$. Then, it follows from the BDG inequality that
		\begin{equation*}
			\mathbb{E} \left[ \Big| \int_0^T (f_n(s) - f(s)) \! \cdot \! \dd X_s \Big| \right] \leq C_{\text{BDG}} \mathbb{E} \left[ \Big| \int_0^T (f_n(s) - f(s))^{\otimes 2} \! : \! \dd \langle X \rangle_s\Big|^{\frac12} \right] \longrightarrow 0 \text{ in } \mathbb{L}^2,
		\end{equation*}
 by dominated convergence, as
\b*
\Big|\! \int_0^T \!\!\! (f_n(s) - f(s))^{\otimes 2} \! : \! \dd \langle X \rangle_s \Big|^{\frac12} 
&\le& 
C \Big(\!\sup_s (1 + |X_s|)^2 |\langle X \rangle_T|\!\Big)^{\frac12} 
\\
&\leq& 
\frac{C}{2} \Big(\!\sup_s (1 + |X_s|)^2 + | \langle X \rangle_T| \!\Big) 
\!\in\! \mathbb{L}^1\!.
\e*
This shows that $\int_0^T \partial_x \delta_m u(\mu_{t^n(s)}, X_s) \! \cdot \! \dd X_s \longrightarrow \int_0^T \partial_x \delta_m u(\mu_{s}, X_s) \! \cdot \! \dd X_s$ in $\mathbb{L}^1$, thus implying by the Jensen inequality that
\b*
U^n_1 
\longrightarrow 
\mathbb{E}^0 \left[ \int_0^T \partial_x \delta_m u(\mu_s, X_s) \! \cdot \! \dd X_s \right]
&\mbox{in}& 
\mathbb{L}^1,
~~\mbox{as}~~n\to\infty,
\e*
and therefore also almost surely along some subsequence. 
\vspace{3mm}

\noindent {\it Step 3.} We next analyse the convergence of the third term
\b*
U^n_3
&\!\!\!\!:=&\!\!\!\!
\sum_{i=1}^{p_n} \int_0^1\!\! \lambda\! \int_0^1\!\! \mathbb{E}^0\hat{\mathbb{E}}^0 \left[  \partial_x \partial_{\hat{x}} \delta_m^2 u(\mu_{t_i^n}^{\lambda \lambda'}, X_{t_{i-1}^n}, \hat{\xi}_{t_{i-1}^n}) \! : \! (\Dpin X_{t_i^n} )(\Dpin \hat{X}_{t_i^n} )^\intercal \right] \dd \lambda' \dd \lambda,
\e*
by arguing in two steps:
\begin{itemize}
\item First, substituting $\mu_{t_{i-1}^n}$ to $\mu_{t_i^n}^{\lambda \lambda'}$, and $\hat{X}_{t_{i-1}^n}$ to $\hat{\xi}_{t_{i-1}^n}$, we compute that
\b*
\overline{U}^n_3
&\!\!\!\!:=&\!\!\!\!
\sum_{i=1}^{p_n} \!\int_0^1 \!\!\!\lambda \! \int_0^1 \!\!\mathbb{E}^0 \hat{\mathbb{E}}^0 \!\!\left[  \partial_x \partial_{\hat{x}} \delta_m^2 u(\mu_{t_{i}^n}, X_{t_{i-1}^n}, \hat{X}_{t_{i-1}^n}) \! : \! (\Dpin X_{t_i^n} )(\Dpin \hat{X}_{t_i^n} )^\intercal \right] \dd \lambda' \dd \lambda \\
\\
&\!\!\!\!=&\!\!\!\!\!
\dfrac{1}{2} \mathbb{E}^0 \hat{\mathbb{E}}^0\!\! \left[ \sum_{i=1}^{p_n}  H_{t_{i-1}^n} \! : \! (\Dpin X_{t_i^n} )(\Dpin \hat{X}_{t_i^n} )^\intercal \right]\!,
\,\mbox{with}\,
H_s \!:=\! \partial_x \partial_{\hat{x}} \delta_m^2 u(\mu_s, X_s, \hat{X}_s),
\e*
which leads by the polarized version of Lemma \ref{lemma} to:
\b*
\overline{U}^n_3
\longrightarrow
\dfrac{1}{2} \mathbb{E}^0 \hat{\mathbb{E}}^0  \left[ \int_0^T  \partial_x \partial_{\hat{x}} \delta_m^2 u (\mu_s, X_s, \hat{X}_s) \! : \! \dd \langle X, \hat{X} \rangle_s \right],
&\mbox{in}&
\L^1,~\mbox{as}~n\to\infty.
\e*			
\item we next control the error $U^n_3\!-\!\overline{U}^n_3\!=\!\mathbb{E}^0 \hat{\mathbb{E}}^0\big[\sum_{i=1}^{p_n}\varepsilon_{t_{i-1}^n} \!\! :\!\! (\Delta^{\pi^n} X_{t_i^n})(\Delta^{\pi^n} \hat{X}_{t_i^n})^\intercal \big]$, with
\b*
\varepsilon_{t_{i-1}^n} 
&\hspace{-3mm}:=& \hspace{-3mm}
\int_0^1\!\!\! \int_0^1 \!\!\!\lambda\! \left(  \partial_x \partial_{\hat{x}} \delta_m^2 u(\mu_{t_i^n}^{\lambda \lambda'}, X_{t_{i-1}^n}, \hat{X}_{t_{i-1}^n}) 
\!-\!  \partial_x \partial_{\hat{x}} \delta_m^2 u(\mu_{t_{i-1}^n}, X_{t_{i-1}^n}, \hat{\xi}_{t_{i-1}^n}) \right) \dd \lambda' \dd \lambda.
\e*
We write the proof for $d = 1$, as the $d$-dimensional case does not raise any difficulty.
\be
\Big| \sum_{i=1}^{p_n} \varepsilon_{t_{i-1}^n} (\Dpin X_{t_i^n} )(\Dpin \hat{X}_{t_i^n} ) \Big| 
&\!\!\!\leq&\!\!\! 
\sum_{i=1}^{p_n} |\Dpin X_{t_i^n} ||\Dpin \hat{X}_{t_i^n} ||\varepsilon_{t_{i-1}^n}| 
\nonumber\\
&\!\!\!\leq&\!\!\! 
\dfrac{1}{2} \sum_{i=1}^{p_n} \Big| \varepsilon_{t_{i-1}^n} \Big| |\Dpin X_{t_i^n} |^2 + \dfrac{1}{2}\sum_{i=1}^{p_n} \Big| \varepsilon_{t_{i-1}^n} \Big| |\Dpin \hat{X}_{t_i^n} |^2
\nonumber\\
&\hspace{-22mm}\leq&\hspace{-13mm} 
\varepsilon_{\pi^n} Q^n,
~\mbox{with}~
Q^n\!:=\! \sum_{i=1}^{p_n} |\Dpin X_{t_i^n} |^2 + \sum_{i=1}^{p_n} |\Dpin \hat{X}_{t_i^n} |^2,~~~~~~~
\label{convterm3}
\ee
and for $i=1,\ldots,n$,
\begin{equation*}
\big| \varepsilon_{t_{i-1}^n} \big| \leq \varepsilon_{\pi^n} 
:= \hspace{-3mm}
\sup_{\tiny\begin{array}{cc}
         0  \leq j \leq n-1
         \\
         t_j^n \leq s_1, s_2 \leq t_{j+1}^n
         \end{array}} 
         \hspace{-3mm}
         \Big|  \partial_x \partial_{\hat{x}} \delta_m^2 u(\mu_{s_1}, X_{s_2}, \hat{X}_{t_j^n}) -  \partial_x \partial_{\hat{x}} \delta_m^2 u(\mu_{t_j^n}, X_{t_j^n}, \hat{X}_{t_j^n})\Big|.
			\end{equation*}
Notice that the map $g:(s_1, s_2, r)\longmapsto \partial_x \partial_{\hat{x}} \delta^2_m u(\mu_{s_1}, X_{s_2}, \hat{X}_r)$ is a.s. continuous on the compact $[0, T]^3$, therefore it is uniformly continuous, and thus
\begin{equation*}
\varepsilon_{\pi^n} = \sup_{0 \leq i \leq n-1} \sup_{t_{i-1}^n \leq s_1, s_2\leq t_i^n} |g(s_1, s_2, t_{i-1}^n) - g(t_{i-1}^n, t_{i-1}^n, t_{i-1}^n)| 
\longrightarrow 0,
~\mbox{a.s.}
\end{equation*}
Since $Q^n\longrightarrow \langle X \rangle_T + \langle \hat{X} \rangle_T < + \infty$ in $\L^1$, we deduce from \eqref{convterm3} together with the dominated convergence theorem (using the fact that $\varepsilon_\pi$ is uniformly bounded, because $\partial_x \partial_{\hat{x}} \delta_m^2 u$ is bounded) that the error term converges towards 0 in $\mathbb{L}^1$, and therefore that it is still the case after taking conditional expectations. This last convergence in $\mathbb{L}^1$ yields the a.s. convergence along a subsequence.
\end{itemize}

\noindent {\it Step 4.}
By following the same line of argument as in Step 3, we also obtain the following convergence in $\L^1$ for the second term
\begin{align*}
\sum_{i=1}^{p_n} \int_0^1 \!\!\mathbb{E}^0\!\! \left[ \partial^2_x \delta_m u(\mu_{t_{i-1}^n}^\lambda, \xi_{t_{i-1}^n}) \! : \! (\Dpin X_{t_i^n} )(\Dpin X_{t_i^n} )^\intercal  \right] \dd \lambda
\longrightarrow 
\mathbb{E}^0 \!\!\left[ \int_0^T \!\!\!\!\partial_x^2 \delta_m(\mu_s, X_s) \! : \! \dd \langle X \rangle_s \right].
\end{align*}
\end{proof}

\section{It\=o\,-Wentzell's formula}
	\label{sec:Ito-Wentzell}

In addition to the continuous semimartingale $X$ of \eqref{def:X} with canonical decomposition 
$$
X=X_0+M+A,
$$
we now consider the extension from a deterministic function $u$ to a process $U : [0, T] \times \mathcal{P}_2(\mathbb{R}^d) \times \Omega \longrightarrow \mathbb{R}$, with the following dynamics for the random field $U_t(m)$:
\be\label{dyn:U}
U_t (m) 
&=& 
U_0(m)
+\int_0^t \phi_r(m) \! \cdot \! \dd D_r + \int_0^t \psi_r(m) \! \cdot \! \dd L_r,
\ee
where $C$ is a finite-variation process and $L$ is a martingale. Our main result is the following It\=o\,-Wentzell formula which will be established under the following conditions.

\begin{assumption}\label{assum:ItoWentzell}
The maps $f \in \left\lbrace U_0,\phi_t, \psi_t,t\in[0,T] \right\rbrace$ defined on $\Pc_2(\R^d)$ satisfy

{\rm(IW1)} $\delta_m f, \partial_x^2 \delta_m f, \delta_m^2 f, \partial_x \partial_{\hat{x}} \delta_m^2 f$ exist and are continuous;

{\rm(IW2)} $f, \partial_x^2 \delta_m f, \partial_x \partial_{\hat{x}} \delta_m^2 f$ are bounded;

\noindent and the processes $X$ together with the driving processes $B,N$ of the random field $U$ satisfy

{\rm(IW3)} $X_0,|A|_{\rm TV}$ and $\langle M \rangle_T$ are square integrable, and both $|D|_{\rm TV}$, $\langle L \rangle_T$ are bounded.
\end{assumption}

We observe that the boundedness condition on $|D|_{\rm TV}$ and $\langle L \rangle_T$ can be weakened at the price of stronger boundedness conditions on $\partial_x\delta_m\phi$ and $\partial_x\delta_m\psi$. We deliberately choose this setup in order to compare to the conditions of dos Reis \& Platonov \cite{reis2022}.

\begin{theorem} \label{thm:ItoWentzell}
Let Assumption \ref{assum:ItoWentzell} hold, and further assume that, for all compact $K \subset \mathcal{P}_2(\mathbb{R}^d)$, we have:
\begin{equation}\label{eq:integrability-random-field}
 \sup_{(t, m) \in [0,T] \times K} \sup_{x, \hat x \in \mathbb{R}^d} \Big\{ | \partial_{x} \delta_m U_t(m,x) | + | \partial_{xx}^2 \delta_m U_t(m,x) | + | \partial_{x\hat x}^2 \delta_{mm}^2 U_t(m,x,\hat x) | \Big\} < \infty, \ \mathbb{P}-\mbox{a.s.} 
 \end{equation}
 Then: \\
\noindent {\rm (i)}
All derivatives $\delta_m U, \partial_x^2 \delta_m U, \delta_m^2 U, \partial_x \partial_{\hat{x}} \delta_m^2 U$ exist, are continuous a.s., and are semimartingales defined by the decomposition for $i, j = 1, \ldots, d$:
\b* 
	\partial_{x_i}\delta_mU_t
&\hspace{-3mm}=&\hspace{-3mm}
\partial_{x_i} \delta_m U_0
+ \int_0^t \! \partial_{x_i}\delta_m \phi_s \! \cdot \! \dd D_s 
+ \int_0^t \partial_{x_i}\delta_m\psi_s \! \cdot \!\dd L_s,
\\
\partial_{x_i, x_j}^2 \delta_m U_t 
&\hspace{-3mm}=&\hspace{-3mm} 
\partial_{x_i, x_j}^2 \delta_m U_0+ \int_0^t \partial_{x_i, x_j}^2 \delta_m \phi_s \! \cdot \! \dd D_s + \int_0^t \partial_{x_i, x_j}^2 \delta_m \psi_s \! \cdot \! \dd L_s
\\
\partial_{x_i, \hat{x}_j}^2\delta^2_mU_t
&\hspace{-3mm}=&\hspace{-3mm}
\partial_{x_i, \hat{x}_j}^2\delta^2_m U_0
+\!\! \int_0^t \!\! \partial_{x_i, \hat{x}_j}^2\delta^2_m\phi_s \! \cdot \! \dd D_s 
\!+ \int_0^t \! \partial_{x_i, \hat{x}_j}^2\delta^2_m\psi_s \! \cdot \! \dd L_s.
\e*
{\rm (ii)} We have, $\mathbb{P}$-a.s.,
\begin{align*}
U_T(\mu_T) \!-\! U_0(\mu_0) 
&= \mathbb{E}^0 \left[ \int_0^T  \partial_x \delta_m U_s(\mu_s, X_s) \! \cdot \! \dd X_s + \dfrac{1}{2}  \partial^2_x \delta_m U_s(\mu_s, X_s) \! : \! \dd \langle X \rangle_s \right] \\
&
+ \mathbb{E}^0 \hat{\mathbb{E}}^0 
   \left[  \int_0^T\dfrac{1}{2}\partial_x \partial_{\hat{x}} \delta_m^2 
                                          U_s (\mu_s, X_s, \hat{X}_s) 
                                          \! : \! \dd \langle X, \hat{X} \rangle_s 
  \right]
\\
&
+\!\!\int_0^T \!\!\!\phi_s(\mu_s) \! \cdot \! \dd D_s 
                    \!+\! \psi_s(\mu_s) \! \cdot \! \dd L_s 
+\mathbb{E}^0 \left[ \int_0^T\!\!\! \partial_x \delta_m \psi_s(\mu_s, X_s) \! : \! \dd \langle L, M \rangle_s \right]\!,
\end{align*}
where $\mathbb{E}^0\big[ \cdot \big] := \mathbb{E}\big[ \cdot | \mathcal{F}_T^0, D, L\big]$ and $\hat{\mathbb{E}}^0\big[ \cdot \big] := \mathbb{E}\big[ \cdot | \mathcal{F}_T^0, D, L, X\big]$. 
\end{theorem}
	
\begin{proof} We organize the proof in several steps.

\noindent {\bf 1.} We start by the existence and continuity of $\delta_m U, \partial_x^2 \delta_m U, \delta_m^2 U, \partial_x \partial_{\hat{x}} \delta_m^2 U$.  We first show that the  functional linear derivative $\delta_mU_t$ exists for all $t\in[0,T]$, and is given by the first expression in (i), i.e.
\be\label{deltamU}
\delta_mU_t(m,x)
&=&
\delta_m U_0(m, x) 
+ \int_0^t \! \delta_m \phi_s (m, x) \! \cdot \! \dd D_s 
+ \int_0^t \delta_m\psi_s(m, x) \! \cdot \! \dd L_s.
\ee
For arbitrary $m, m' \in \mathcal{P}_2(\mathbb{R}^d)$ with barycenter $m^\lambda:=(1-\lambda)m+\lambda m'$, it follows from the decomposition of $U$ and the definition of the linear functional derivative for the maps $U_0,\phi_s,\psi_s$, as guaranteed by Assumption (IW1), that:
\b*
U_t(m') \!-\! U_t(m) 
&\hspace{-3mm}=& \hspace{-3mm}
U_0(m') \!-\! U_0(m) \!+\!\! \int_0^t \!\!(\phi_s(m') \!-\! \phi_s(m)) \! \cdot \! \dd D_s 
\!+\!\! \int_0^t \!\!( \psi_s(m') \!-\! \psi_s(m)) \! \cdot \!\dd L_s
\\
&\hspace{-40mm}=& \hspace{-22mm}
\!\int_0^1\hspace{-2mm} \int\!\! \delta_m U_0(m^\lambda, x) (m' \!-\! m)(\dd x) \dd \lambda
			\!+\! \int_0^t \!\!\left( \int_0^1\hspace{-2mm} \int \delta_m \phi_s (m^\lambda, x) (m' \!-\! m) (\dd x) \dd \lambda \right) \! \cdot \! \dd D_s \\
\\
&&\hspace{32mm}
+\! \int_0^t \!\!\left( \int_0^1\hspace{-2mm} \int \delta_m \phi_s (m^\lambda, x) (m' \!-\! m) (\dd x) \dd \lambda \right) \! \cdot \! \dd L_s,
\e*
By continuity and boundedness of $\phi$, $\psi$ and their derivatives, we may apply stochastic Fubini theorem (see \cite{protter2012stochastic}), which provides: $
U_t(m') - U_t(m)
= 
\int_0^1\!\!\int F_t(m^\lambda,x)(m' - m)(\dd x) \dd \lambda
$, where $F_t(m,x)$ is given by the right hand side of \eqref{deltamU}.
Moreover, it follows from (IW2) that the maps $\delta_m U_0, \delta_m \phi_s, \delta_m\psi_s$ have linear growth in $x$, uniformly in $m$. As $\E[|D|^2_{\rm TV}+\langle L\rangle_T]<\infty$ by (IW3), this implies that the map $F_t(m,x)$ also has linear growth in $x$, uniformly in $m$. Notice also that 
\begin{itemize}
\item $F_t$ inherits the continuity of $\delta_mU_0$, $\delta_m\phi_t$, and $\delta_m\psi_t$, by the dominated convergence theorem due to their boundedness, uniformly in $t,m$, assumed in (IW2). We may then conclude that $\delta_mU_t=F_t$ by the definition of the linear functional derivative. 
\item $\partial^2_{x}\delta_mU$ exists and inherits the continuity of $\partial^2_{x}\delta_mU_0$, $\partial^2_{x}\delta_m\phi$, and $\partial^2_{x}\delta_m\psi$, by the dominated convergence theorem due to their boundedness, uniformly in $t,m$, assumed in (IW2)
\end{itemize}

Finally, observe that the coefficients of the SDEs driving $U$ and $\delta_mU$ (with fixed $x$) satisfy the same conditions in (IW1) and (IW2). Applying the previous argument to the process $\delta_mU_t(m,x)$, for fixed $x$, it follows that $\delta^2_mU_t$ and $\partial_x\partial_{\hat x}\delta^2_mU_t(m,x,\hat x)$ also exist and are continuous, with decomposition given by the third expression in \rm{(i)}. 

\noindent {\bf 2.}
	Let $\pi^n : 0 = t_0^n < t_1^n < \ldots < t_{p_n}^n = T$ be a dense sequence of partitions of $[0,T]$. As in the proof of It\=o's formula, we start from the telescopic decomposition:
	\begin{align}
		U_{t_i^n} (\mu_{t_i^n}) - U_{t_{i-1}^n} (\mu_{t_{i-1}^n} ) = R_1 + R_2 ;~~&R_1 := U_{t_i^n} (\mu_{t_i^n}) - U_{t_i^n} (\mu_{t_{i-1}^n} ),
		 \label{Telescop-IW}
		 \\
		&R_2 := U_{t_i^n} (\mu_{t_{i-1}^n} ) \!-\! U_{t_{i-1}^n} (\mu_{t_{i-1}^n} )
		           =\! \int_{t_{i-1}^n}^{t_i^n}\!\!\! \dd U_s(\mu_{t_{i-1}^n}),
		\nonumber
	\end{align}
with dynamics of $\{U_s(m),s\ge 0\}$ given by \eqref{dyn:U}. We next further compute $R_1$ by using the definition of the functional linear derivative:
$$
R_1 
\!=\!\!
\int_0^1\!\!\!\! \int \!\!\delta_m U_{t_i^n} (\mu_{t_{i-1}^n}^\lambda, .)
\!\dd (\mu_{t_i^n} \!-\! \mu_{t_{i-1}^n} )\!\dd \lambda
\!=\!\!
\int_0^1\!\! \mathbb{E}^0\! \big[ \delta_m U_{t_i^n} (\mu_{t_{i}^n} ^\lambda, \!X_{t_i^n}) \!-\!  \delta_m U_{t_i^n} (\mu_{t_{i}^n} ^\lambda, \!X_{t_{i-1}^n}) \big] \dd \lambda.
$$
By the second order Taylor theorem, we may rewrite this as
$$
R_1
\!=\!\! \int_0^1\!\! \mathbb{E}^0 \left[ \partial_x \delta_m U_{t_i^n} (\mu_{t_{i}^n}^\lambda, X_{t_{i-1}^n}) \! \cdot \!\Dpin X_{t_i^n} 
\!+\! 
\dfrac{1}{2} \partial^2_x \delta_m U_{t_i^n}  (\mu_{t_{i}^n}^\lambda, \xi_{t_{i-1}^n}) \! : \!(\Dpin X_{t_i^n})^{\otimes 2} \right] \dd \lambda,
$$
for some r.v. $\xi_{t_{i-1}^n}$ lying between $X_{t_{i-1}^n}$ and $X_{t_i^n}$. Denoting $\gamma_s:=\partial_x \delta_m U_s$, we compute that
\b*
\Big\{\gamma_{t^n_i}(., X_{t_{i-1}^n})
\Big\}_{_{\mu_{t_{i-1}^n}^\lambda}}^{^{\mu_{t_{i-1}^n}}}
&=&
\int_0^1\!\!\!\int\!\! \delta_m \gamma_{t^n_i}(\mu_{t_{i}^n}^{\lambda \lambda'}, X_{t_{i-1}^n}, \hat{x}) (\mu_{t_{i}^n}^\lambda \!-\! \mu_{t_{i-1}^n})(\dd \hat{x}) \dd \lambda' 
\\
&=& 
\lambda \int_0^1 \hat{\mathbb{E}}^0 \left[ \left\lbrace \delta_m \gamma_{t^n_i}(\mu_{t_{i}^n}^{\lambda \lambda'}, X_{t_{i-1}^n}, \cdot) \right\rbrace_{\hat{X}_{t_{i-1}^n}}^{\hat{X}_{t_i^n}}   \right] \dd \lambda' \\
&=&
\lambda \int_0^1 \hat{\mathbb{E}}^0 \left[ \partial_{\hat{x}} \delta_m \gamma_{t^n_i}(\mu_{t_{i}^n}^{\lambda \lambda'}, X_{t_{i-1}^n}, \hat{\xi}_{t_{i-1}^n}) \Dpin \hat{X}_{t_i^n} \right] \dd \lambda',
\e*
for some $\hat{\xi}_{t_{i-1}^n}$ between $\hat{X}_{t_{i-1}^n}$ and $\hat{X}_{t_i^n}$. By the regularity results obtained in Step 1 of this proof, we may also write $\gamma_{t^n_i}(\mu_{t_{i-1}^n}, X_{t_{i-1}^n}) = \gamma_{t^n_{i-1}}(\mu_{t_{i-1}^n}, X_{t_{i-1}^n}) + \int_{t_{i-1}^n}^{t_i^n} \dd \gamma_s (\mu_{t_{i-1}^n}, X_{t_{i-1}^n})$. Substituting back the expression of the map $\gamma$, this provides:
\b*
R_1 
&=& 
\mathbb{E}^0 \left[ \partial_x \delta_m U_{t}(\mu_{t_{i-1}^n}, X_{t_{i-1}^n}) \! \cdot \! \Dpin X_{t_i^n} \right] 
+ \mathbb{E}^0 \left[ \Big(\int_{t_{i-1}^n}^{t_i^n} \dd \gamma_s (\mu_{t_{i-1}^n}, X_{t_{i-1}^n})\Big) \! \cdot \! \Dpin X_{t_i^n} \right] \\
\\
&&
+\; \mathbb{E}^0 \hat{\mathbb{E}}^0 \left[ \int_0^1\!\!\! \int_0^1 \lambda \partial_{\hat{x}} \delta_m \partial_x \delta_m U_{t_i^n}  (\mu_{t_{i}^n}^{\lambda \lambda'}, X_{t_{i-1}^n}, \hat{\xi}_{t_{i-1}^n}) \! : \! (\Dpin X_{t_i^n}) (\Dpin \hat{X}_{t_i^n})^\intercal \dd \lambda' \dd \lambda  \right] 
\\
&&+ \int_0^1 \mathbb{E}^0 \left[ \dfrac{1}{2} \partial^2_x \delta_m U_{t_i^n}  (\mu_{t_{i}^n}^\lambda, \xi_{t_{i-1}^n}) \! : \! (\Dpin X_{t_i^n})(\Dpin X_{t_i^n})^\intercal  \right] \dd \lambda,
\e*
where $\dd \gamma_s (\mu_{t_{i-1}^n}, X_{t_{i-1}^n})
=
f_s(\mu_{t_{i-1}^n}, X_{t_{i-1}^n})\dd D_s+g_s(\mu_{t_{i-1}^n}, X_{t_{i-1}^n})\dd L_s$,
$s\in[t^n_{i-1}, t^n_i)$, with
\b*
(f_s,g_s) &:=&\partial_x \delta_m(\phi_s,\psi_s).
\e*
Summing the decomposition \eqref{Telescop-IW}, and denoting by $t^n(s)$ the closest subdivision point strictly to the left of $s$, this provides:
\b*
U_T(\mu_T) - U_0(\mu_0)
&=& 
\int_{0}^{T} \dd U_s (\mu_{t^n(s)}) 
+\int_{0}^{T} \mathbb{E}^0 \left[ \partial_x \delta_m U_{t^n(s)}(\mu_{t^n(s)}, X_{t^n(s)}) \! \cdot \! \dd X_s \right] 
\\
&& \hspace{-10mm}
+  \sum_{i=1}^{p_n} 
    \mathbb{E}^0 \Big[ 
                                  \Big(
                                  \int_{t_{i-1}^n}^{t_i^n} f_s(\mu_{t_{i-1}^n}, X_{t_{i-1}^n})\dd D_s
                                                                  +g_s(\mu_{t_{i-1}^n}, X_{t_{i-1}^n})\dd L_s  
                                 \Big) \! \cdot \! \Dpin X_{t_i^n}
                                 \Big] 
\\
&& \hspace{-28mm}
+ \sum_{i=1}^{p_n} \mathbb{E}^0 \hat{\mathbb{E}}^0\!\!
   \left[ \int_0^1\!\!\!\int_0^1 \lambda \partial_{\hat{x}} \delta_m \partial_x \delta_m U_{t_i^n} (\mu_{t_{i}^n}^{\lambda \lambda'}, X_{t_{i-1}^n}, \hat{\xi}_{t_{i-1}^n}) \! : \! (\Dpin X_{t_i^n} ) (\Dpin \hat{X}_{t_i^n} )^\intercal \dd \lambda' \dd \lambda  \right] 
\\
&& \hspace{-28mm}
+\sum_{i=1}^{p_n} \int_0^1 \mathbb{E}^0\!\! \left[ \dfrac{1}{2} \partial^2_x \delta_m U_{t_i^n} (\mu_{t_{i}^n}^\lambda, \xi_{t_{i-1}^n}) \! : \! (\Dpin X_{t_i^n} )(\Dpin X_{t_i^n} )^\intercal  \right] \dd \lambda.
\e*	
{\bf 3.} We now show that the different terms of the last decomposition converge towards the formula announced in the theorem, for a certain dense sequence of subdivisions. 

{\bf 3.1} We first prove that, after possibly passing to a subsequence,
\be\label{ItoWentzel-1}
\int_{0}^{T} \dd U_s (\mu_{t^n(s)})  
=\int_0^T \phi_s (\mu_{t^n(s)}) \! \cdot \! \dd D_s+\psi_s (\mu_{t^n(s)}) \! \cdot \! \dd L_s
\longrightarrow
\int_0^T \dd  U_s(\mu_s) \text{ a.s.}
\ee
By the dominated convergence theorem for the Stieltjes integral $\int \cdot \dd D_s$, and the boundedness of $\phi$, we obtain the convergence of the finite variation part
\b*
\int_0^T \phi_s (\mu_{t^n(s)}) \! \cdot \! \dd D_s 
&\longrightarrow& 
\int_0^T \phi_s (\mu_s) \! \cdot \! \dd D_s, \text{ a.s.}
\e*
As for the stochastic integral component, we estimate by the It\=o isometry that
\be\label{proofiwv3}
\mathbb{E} \left[  \left(  \int_0^T  (\psi_s(\mu_{t^n(s)}) - \psi_s(\mu_s)) \! \cdot \! \dd L_s 
                           \right)^2
                   \right]
\!\!\!\!&\le&\!\!\!\!
\mathbb{E} \left[  \int_0^T     (\psi_s(\mu_{t^n(s)}) - \psi_s(\mu_s))^{\otimes 2} \! : \!\dd \langle L \rangle_s \right]\!,~~~~
\ee
Since $\psi$ and $\mu$ are a.s. continuous and $\psi$ is bounded, it follows from dominated convergence for the Stieltjes stochastic integral $\int \cdot \dd \langle N \rangle_s$ that $\int_0^T\big(\psi_s(\mu_{t^n(s)}) - \psi_s(\mu_s)\big)^{\otimes 2} \! : \! \dd \langle N \rangle_s \longrightarrow 0$, a.s. Furthermore, $| \int_0^T\big(\psi_s(\mu_{t^n(s)}) - \psi_s(\mu_s)\big)^{\otimes 2} \!:\!\dd \langle L \rangle_s|\leq 4 || \psi ||_{\infty}^2  \mbox{Tr} [\langle L \rangle_T] $, which is in $\mathbb{L}^1$. We then deduce from \eqref{proofiwv3} and the dominated convergence theorem that
\b*
\int_0^T  \psi_s(\mu_{t^n(s)}) \! \cdot \! \dd L_s 
&\longrightarrow& 
\int_0^T \psi_s(\mu_s) \! \cdot \! \dd L_s \text{ in } \L^2,
~\mbox{and a.s. along some subsequence,}
\e*
thus completing the proof of \eqref{ItoWentzel-1}.

For the remaining terms, we use the same method as in the proof of Theorem \ref{thm:Ito}, by arguing that the sequence inside the conditional expectations converges in $\mathbb{L}^1$ towards the desired results.

{\bf 3.2.} Denote $H_s := \partial_x \delta_m U_s (\mu_s, X_s)$. The convergence of the second term is implied by the following two convergence results: 
\be
\int_0^T H_{t^n(s)} \! \cdot \! \dd A_s \longrightarrow \int_0^T H_s \! \cdot \! \dd A_s,
&\mbox{and}&
\int_0^T H_{t^n(s)} \! \cdot \! \dd M_s
\longrightarrow
\int_0^T H_s \! \cdot \! \dd M_s, 
~~\mbox{in}~\L^1.
\ee
The first convergence of the finite variation part follows from the a.s. pathwise continuity of the process $H$, together with the dominated convergence theorem for the Stieltjes integral $\int \cdot \dd A_s$ together with the linear growth of $\partial_x \delta_m U_s$ in the $x-$variable, as implied by (IW2). Similarly, it follows from the BDG inequality and the dominated convergence theorem for the Stieltjes integral $\int \cdot \dd \langle M \rangle_s$ that $\int_0^TH_{t^n(s)} \! \cdot \! \dd M_s \longrightarrow \int_0^TH_s \! \cdot \! \dd M_s$ in $\L^1$. 

{\bf 3.3.} In this step, we justify the $\L^1$ convergence of the third term. Denoting $F_t:=\int_0^t f_s(\mu_s, X_s) \dd D_s$, and $G_t\!:=\! \int_0^t \!\!g_s(\mu_s, X_s) \dd L_s$, we shall now show that
\b*
\Phi^n
:=
\sum_{i=1}^{p_n} 
\int_{t_{i-1}^n}^{t_i^n} \!\!f_s(\mu_{t_{i-1}^n}, X_{t_{i-1}^n}) \dd D_s \cdot \Dpin \!X_{t_i^n}
&\longrightarrow&
\mbox{Tr}[\langle X,F \rangle_T]=0,~\mbox{in}~\L^1\!,~\mbox{as}~n\to\infty
\\
\Psi^n
:=
\sum_{i=1}^{p_n} 
\int_{t_{i-1}^n}^{t_i^n} \!\!g_s(\mu_{t_{i-1}^n}, X_{t_{i-1}^n}) \dd L_s \cdot \Dpin \!X_{t_i^n} 
&\longrightarrow&
\mbox{Tr}[\langle X,G \rangle_T], ~\mbox{in}~\L^1\!,~\mbox{as}~n\to\infty,
\e*
where we have $\langle F, X \rangle_T=0$, a.s. due to the fact that the process $F$ has finite variation. In order to prove the first convergence, we use Lemma \ref{lemma} to conclude that, along some subsequence,
\b*
\sum_{i=1}^{p_n} (\Dpin X_{t_i^n} )\cdot(\Dpin F_{t_i^n} )  
\longrightarrow 
\mbox{Tr}[\langle X, F \rangle_T], \text { in } \mathbb{L}^1.
\e*
Denote $\delta^n_s := f_s(\mu_{t^n(s)}, X_{t^n(s)}) - f_s(\mu_s, X_s)$. We have, by Cauchy-Schwarz inequality: 
\b*
\E\Big|\Phi^n-\sum_{i=1}^{p_n} (\Dpin\! X_{t_i^n} ) \! \cdot \! (\Dpin\! F_{t_i^n} )\Big| 
&=&
\E\Big|\sum_{i=1}^{p_n} (\Dpin\! X_{t_i^n}) \! \cdot \! \int_{t_{i-1}^n}^{t_i^n} \delta^n_s
\dd D_s\Big|
\\ &\le&
C\mathbb{E}\Big[\sum_{i=1}^{p_n} | \Dpin\! X_{t_i^n} |^2 \Big]^{1/2}
                           \mathbb{E}\Big[\sum_{i=1}^{p_n} \Big| \int_{t_{i-1}^n}^{t_i^n} \delta^n_s
\dd D_s \Big|^2
                           \Big]^{1/2}.
\\
&\le&
C\mathbb{E}\Big[\sum_{i=1}^{p_n} | \Dpin\! X_{t_i^n} |^2 \Big]^{1/2}
                           \mathbb{E}\Big[\int_0^T |\delta^n_s|^2 \dd |D|_s
                           \Big]^{1/2}.
\e*
By boundedness of $\sum_{i=1}^{p_n} | \Dpin\! X_{t_i^n} |^2$ in $\mathbb{L}^1$, boundedness of $\delta^n$, Conditions (IW3) and the fact that $\delta^n_s \to 0$, $\mathrm{d}s \otimes \mathrm{d}\mathbb{P}$ a.e, as $n \to \infty$, we conclude by the dominated convergence theorem that:
$$\E\Big|\Phi^n-\sum_{i=1}^{p_n} (\Dpin X_{t_i^n} ) \! \cdot \! (\Dpin F_{t_i^n} )\Big| \underset{n \to \infty}{\longrightarrow} 0. $$
		
A similar argument allows to justify the $\L^1$ convergence of $\Psi^n$ towards $\mbox{Tr}[\langle G, X \rangle_T]$ in  $\mathbb{L}^1$. Indeed, using again Lemma \ref{lemma}, we are reduced to the following estimate involving the process $\eta^n_s := g_s(\mu_{t^n(s)}, X_{t^n(s)}) - g_s(\mu_s, X_s)$: 
\b*
\E\Big|\Psi^n-\sum_{i=1}^{p_n} (\Dpin\! X_{t_i^n} ) \! \cdot \! (\Dpin\! G_{t_i^n} )\Big| 
&=&
\E\Big|\sum_{i=1}^{p_n} (\Dpin\! X_{t_i^n}) \! \cdot \! \int_{t_{i-1}^n}^{t_i^n} \eta^n_s
\dd L_s\Big|
\\ &\le&
C\mathbb{E}\Big[\sum_{i=1}^{p_n} | \Dpin\! X_{t_i^n} |^2 \Big]^{1/2}
                           \mathbb{E}\Big[\sum_{i=1}^{p_n} \Big| \int_{t_{i-1}^n}^{t_i^n} \eta^n_s
\dd L_s \Big|^2
                           \Big]^{1/2}.
\\
&=&
C\mathbb{E}\Big[\sum_{i=1}^{p_n} | \Dpin\! X_{t_i^n} |^2 \Big]^{1/2}
                           \mathbb{E}\Big[\int_0^T |\eta^n_s|^2 \dd \langle L \rangle_s
                           \Big]^{1/2}.
\e*
by the Cauchy-Schwartz inequality and the It\=o isometry. Similarly to $\Phi^n$, we conclude by the dominated convergence theorem that:
$$\E\Big|\Psi^n-\sum_{i=1}^{p_n} (\Dpin X_{t_i^n} ) \! \cdot \! (\Dpin G_{t_i^n} )\Big| \underset{n \to \infty}{\longrightarrow} 0. $$

{\bf 3.4.} We finally prove that, a.s.:
		\be\label{eq:cv-second-1}
			\mathbb{E}^0\Big[\sum_{i=1}^{p_n} \int_0^1 \dfrac{1}{2} \partial^2_x \delta_m U_{t_i^n} (\mu_{t_{i}^n}^\lambda, \xi_{t_{i-1}^n}) \! : \! (\Dpin X_{t_i^n} )(\Dpin X_{t_i^n} )^\intercal  \dd \lambda \Big]  \\\longrightarrow \mathbb{E}^0\Big[ \dfrac{1}{2} \int_0^T \partial^2_x \delta_m u_s (\mu_s, X_s) \! : \! \dd \langle X \rangle_s \Big] \nonumber
		\ee
		and
		\be\label{eq:cv-second-2}
			\mathbb{E}^0\Big[\sum_{i=1}^{p_n}  \int_0^1  \int_0^1  \lambda  \partial_{\hat{x}}  \delta_m  \partial_x  \delta_m U_{t_i^n}  (\mu_{t_{i}^n}^{\lambda \lambda'},  X_{t_{i-1}^n},  \hat{\xi}_{t_{i-1}^n} ) \! : \! (\Dpin X_{t_i^n} )(\Dpin \hat{X}_{t_i^n} )^\intercal \dd  \lambda'  \dd  \lambda \Big] \\ 
			\longrightarrow \mathbb{E}^0\Big[\dfrac{1}{2}  \int_0^T  \partial_{\hat{x}}  \delta_m  \partial_{x}   \delta_m U_s (  \mu_s , X_s,   \hat{X}_s)   \! : \! \dd \langle X,   \hat{X} \rangle_s\Big] \nonumber
		\ee
		We only write the detailed argument for \eqref{eq:cv-second-1}, as it is similar for \eqref{eq:cv-second-2}. Introducing $ \varepsilon_{t_i^n} := \partial^2_x \delta_m U_{t_i^n} (\mu_{t_{i}^n}^\lambda, \xi_{t_{i-1}^n}) - \partial^2_x \delta_m U_{t_i^n} (\mu_{t_{i-1}^n}, \xi_{t_{i-1}^n})$, we have:
		\begin{align*}
		 \sum_{i=1}^{p_n} \int_0^1 &\dfrac{1}{2} \partial^2_x \delta_m U_{t_i^n} (\mu_{t_{i}^n}^\lambda, \xi_{t_{i-1}^n}) \! : \! (\Dpin X_{t_i^n} )(\Dpin X_{t_i^n} )^\intercal  \dd \lambda \\ =& \sum_{i=1}^{p_n}  \dfrac{1}{2} \partial^2_x \delta_m U_{t_i^n} (\mu_{t_{i-1}^n}, \xi_{t_{i-1}^n}) \! : \! (\Dpin X_{t_i^n} )(\Dpin X_{t_i^n} )^\intercal   \\ &+  \sum_{i=1}^{p_n} \int_0^1 \dfrac{1}{2} \varepsilon_{t_i^n} \! : \! (\Dpin X_{t_i^n} )(\Dpin X_{t_i^n} )^\intercal  \dd \lambda.
		 \end{align*}
		 By Lemma \ref{lemma}, the first term in the right hand side member of the equality tends to the desired quantity as $n \to \infty$. We then have to show that the last term goes to $0$. Observe that:
		 \begin{align*}
		\Big| \sum_{i=1}^{p_n} \int_0^1 \dfrac{1}{2} \varepsilon_{t_i^n} \! : \! (\Dpin X_{t_i^n} )(\Dpin X_{t_i^n} )^\intercal  \dd \lambda \Big| \le \Big( \int_0^1 \max_{1 \le i \le n} | \varepsilon_{t_i^n} | \dd \lambda \Big) \sum_{i=1}^n | \Dpin X_{t_i^n} |^2.
		 \end{align*}
		 For the same reason as in the proof of Theorem \ref{thm:Ito}, Step 3, we have $\int_0^1 \max_{1 \le i \le n} | \varepsilon_{t_i^n} | \dd \lambda \to 0$ as $n \to \infty$, a.s. Moreover, $\sum_{i=1}^n | \Dpin X_{t_i^n} |^2$ converges in $\mathbb{L}^1$ by Lemma \ref{lemma} and therefore is bounded up to some subsequence. Therefore, passing to this subsequence, we have 
		$$
		  \sum_{i=1}^{p_n} \int_0^1 \dfrac{1}{2} \varepsilon_{t_i^n} \! : \! (\Dpin X_{t_i^n} )(\Dpin X_{t_i^n} )^\intercal  \dd \lambda \underset{n \to \infty}{\longrightarrow} 0, \quad \mbox{a.s.} 
		  $$
		 Furthermore, by the assumption \eqref{eq:integrability-random-field}, $\max_{1 \le i \le n} | \varepsilon_{t_i^n} |$ is bounded by a random variable $Z$ which is $\mathcal{F}_T^0 \vee \sigma(D, L)$-measurable. Therefore, we have
		 $$ \Big| \sum_{i=1}^{p_n} \int_0^1 \dfrac{1}{2} \varepsilon_{t_i^n} \! : \! (\Dpin X_{t_i^n} )(\Dpin X_{t_i^n} )^\intercal  \dd \lambda \Big| \le Z \sum_{i=1}^n | \Dpin X_{t_i^n} |^2,$$
		  which is uniformly integrable for the conditional expectation $\mathbb{E}^0$, which finally implies that 
		 $$  \mathbb{E}^0\Big[\sum_{i=1}^{p_n} \int_0^1 \dfrac{1}{2} \varepsilon_{t_i^n} \! : \! (\Dpin X_{t_i^n} )(\Dpin X_{t_i^n} )^\intercal  \dd \lambda \Big] \underset{n \to \infty}{\longrightarrow} 0, \quad \mbox{a.s.} $$
\end{proof}
\subsection{Examples}
\subsubsection{Brownian Case}
Let us consider the special case where the process $X$ and the random field $U$ are It\=o processes defined by
\b*
\dd U_t (m) 
&=& 
\phi_t (m) \dd t + \psi_t(m) \cdot \dd W_t + \psi_t^0(m) \cdot \dd W_t^0, 
\\
\dd X_t &=& 
b_t \dd t + \sigma_t \dd W_t + \sigma_t^0 \dd W_t^0
\e*
where $W$, $W^0$ are Brownian motions. We choose here $\mathbb{F}^0 = (\mathcal{F}^0_t)_{0 \leq t \leq T}$ to be the filtration generated by $W^0$. This setting reduces to that of dos Reis \& Platonov \cite{reis2022}. The conditionally independent copy $\hat{X}$ is defined by
\b*
\dd \hat{X}_t 
&=& 
\hat{b}_t \dd t + \hat{\sigma}_t  \dd \hat{W}_t + \hat{\sigma}^0_t \dd W_t^0
\e*
where $\hat{b}, \hat{\sigma}, \hat{\sigma}^0, \hat{W}$ are conditionally independent copies of $b, \sigma, \sigma^0, W$, respectively. We rephrase our Theorem \ref{thm:ItoWentzell} in the present setting in order to compare it with the corresponding statement in \cite{reis2022}.

\begin{corollary}
For $f \in \left\lbrace U_0,\phi_t, \psi_t, \psi^0_t,t\in[0,T] \right\rbrace$, assume:

$\bullet$ $f, \delta_m f, \partial_x^2 \delta_m f, \delta_m^2 f, \partial_{\hat{x}}^2 \delta_m^2 f, \partial_x \partial_{\hat{x}} \delta_m^2 f$ exist and are continuous;

$\bullet$ $\partial_x^2 \delta_m f, \partial_{\hat{x}}^2 \delta_m^2 f, \partial_x \partial_{\hat{x}} \delta_m^2 f$ are bounded;

$\bullet$ $\mathbb{E}\left[ \int_0^T (|X_0|^2+| b_s|^2 + |\sigma_s \sigma_s^\intercal|^2 + |\sigma_s^0 (\sigma_s^0)^\intercal|^2 )\dd s  \right] < \infty$.
\\
We also assume that $U$ satisfies \eqref{eq:integrability-random-field}. Then $\delta_m U, \partial_x^2 \delta_m U, \delta_m^2 U, \partial_x \partial_{\hat{x}} \delta_m^2 U$ exist, are continuous a.s., and are It\=o processes driven by the Brownian motions $W$ and $W^0$, with coefficients defined by the corresponding derivatives of the coefficients of $U$. Moreover, we have:
\b*
U_T(\mu_T) - U_0(\mu_0) 
&\!\!=&\!\! \!\!
\int_0^T \phi_s(\mu_s) \dd s + \int_0^T \psi_s(\mu_s) \! \cdot \! \dd W_s + \int_0^T \psi_s^0(\mu_s) \! \cdot \! \dd W_s^0 
\\
&\!\!&\!\!\!\!
+ \mathbb{E}^0 \left[ \int_0^T  \partial_x \delta_m U_s(\mu_s, X_s) \! \cdot \! b_s \dd s + \int_0^T  (\sigma_s^0 )^\intercal \partial_x \delta_m U_s(\mu_s, X_s) \! \cdot \! \dd W_s^0 \right] 
\\
&\!\!&\!\!\!\!
+ \dfrac{1}{2}  \mathbb{E}^0 \left[ \int_0^T \partial^2_x \delta_m U_s(\mu_s, X_s) \! : \! (\sigma_s \sigma_s^\intercal + \sigma^0_s (\sigma^0_s)^\intercal) \dd s \right] 
\\
&\!\!&\!\!\!\!
+\mathbb{E}^0 \left[\int_0^T \partial_x \delta_m \psi_s^0 (\mu_s, X_s) \! : \! (\sigma_s^0)^\intercal \dd s \right] 
\\
&\!\!&\!\!\!\!
+ \dfrac{1}{2}  \mathbb{E}^0 \hat{\mathbb{E}}^0 \left[  \int_0^T\partial_x \partial_{\hat{x}} \delta_m^2 U_s (\mu_s, X_s, \hat{X}_s) \! : \! \sigma_s^0 (\hat{\sigma}_s^0)^\intercal \dd s \right], \mbox{ a.s.}
\e*
\end{corollary}

We thus find the same formula as in \cite{reis2022}, with the only additional hypothesis that the highest-order derivatives are bounded instead of square-integrable. By standard density argument as in Carmona \& Delarue \cite{CD1,CD2}, one may relax this boundedness condition, however we refrain from this additional development for presentation simplicity. 

\subsubsection{Semimartingale factor random field model}

Suppose that $X$ is a continuous semimartingale:
\begin{equation*}
X_t =X_0 + A_t + M_t \text{ for all } t \in [0, T],
\end{equation*}
where $(A_t)_t$ is a finite-variation process and $(M_t)_t$ is a martingale. In this section, we consider the case where the random field is defined by $U_t(m):=u(t,m,Y_t)$ for some factor process $Y = (Y_t)_{0 \leq t \leq T}$ which is another continuous semimartingale :
\begin{equation*}
Y_t = Y_0 + V_t + S_t,
\end{equation*}
with finite-variation process $V$, and a martingale $S$. Here, the deterministic function $u : (t, m, y) \in [0, T] \times \mathcal{P}_2(\mathbb{R}^d)\times \mathbb{R} \mapsto u(t, m, y) \in \mathbb{R}$ will be assumed to be sufficiently smooth. For this, we extend naturally the definition of the functional linear derivative by reducing to the standard definition once the variables $(t,y)$ are frozen. The second order functional linear derivative $\delta_m^2$ is also defined similarly. The following result is a direct restatement of Theorem \ref{thm:ItoWentzell} in the present context.

\begin{corollary} 
Let us suppose that

$\bullet$ $\partial_t u, \partial_y^2 u, \delta_m u, \partial_x^2 \delta_m u, \delta_m^2 u, \partial_x \partial_{\hat{x}} \delta_m^2 u, \delta_m \partial_y u, \partial_x \delta_m \partial_y u$ exist and continuous;

$\bullet$ $\partial_y^2 u, \partial_x^2 \delta_m u, \partial_x \partial_{\hat{x}} \delta_m^2 u, \partial_x \delta_m \partial_y u$ are bounded ;

$\bullet$ $X_0, |A|_{\rm{TV}}, |V|_{\rm{TV}}, \langle M \rangle_T$ and $S$ are square integrable.
\\
Then, denoting $\Theta_t = (t, \mu_t, Y_t)$ for all $t \in [0, T]$, we have:
\b*
u(\Theta_T) - u(\Theta_0) 
&=&
 \int_0^T \partial_t u(\Theta_s)\dd s + \partial_y u(\Theta_s)\!\cdot\!\dd Y_s + \dfrac{1}{2} \partial^2_y u(\Theta_s)\!:\!\dd \langle Y \rangle_s  
 \\
&&\hspace{-35mm}
+ \mathbb{E}^0 \!\!\left[ \int_0^T\!\!\!  \partial_x \delta_m u(\Theta_s, X_s)\!\cdot\!\dd X_s 
\!+\! \dfrac{1}{2}  \partial^2_x \delta_m u(\Theta_s, X_s)\!:\!\dd \langle X \rangle_s 
\!+\! \partial_x \delta_m \partial_y u(\Theta_s, X_s)\!:\!\dd \langle X, Y \rangle_s \right] 
\\
&&\hspace{-35mm}
+ \dfrac{1}{2}  \mathbb{E}^0 \hat{\mathbb{E}}^0\!\!\left[  \int_0^T\!\!\!\partial_x \partial_{\hat{x}} \delta_m^2 u (\Theta_s, X_s, \hat{X}_s)\!:\!\dd \langle X, \hat{X} \rangle_s \right] \text{ a.s.}
\e*
\end{corollary}
			
\begin{proof}
By the standard It\=o formula for finite-dimension It\=o processes, the random field $\big\{U_t(m) = u_t(t, m, Y_t)$, $(t,m) \in [0, T]\times\mathcal{P}_2(\mathbb{R}^d)\big\}$ has the following semimartingale decomposition:
\b*
\dd U_t (m) 
&=& 
\phi_t(m) \cdot\dd B_t + \psi_t(m) \cdot\dd N_t
\e*
with $B_t = (t, V_t,\langle S \rangle_t)$, $N_t = S_t$, $\phi_t(m) = \big(\partial_t u_t(m, Y_t), \partial_y u_t(m, Y_t),\frac{1}{2} \partial_y^2 u_t(m, Y_t)\big)$ and $\psi_t(m) = \partial_y u_t(m, Y_t)$.
The result is now a direct application of Theorem \ref{thm:ItoWentzell}. 
\end{proof}

\section{Application to Mean-Field Control: HJB Equation}
	\label{sec:control}
Let $\Omega = \mathcal{C}^0(\mathbb{R}_{+}, \mathbb{R}^d)^2 \times \mathcal{C}^0(\mathbb{R}_{+}, \mathbb{R}^{d_0})$ with canonical process $(X_t, Y_t, W_t^0) : (\omega, \omega^0) \in \Omega \mapsto (\omega, \omega^0)(t) \in \mathbb{R}^d \times \mathbb{R}^{d_0}$. The corresponding canonical filtration is denoted by $\F=\{\Fc_t,t\ge 0\}$. We also introduce the control space $\mathcal{A}$ consisting of all $\F-$progressively measurable processes $\alpha$ with values in a compact subset $A$ of a finite dimensional space.
		
Let $b,\sigma$, and $\sigma^0$ be given bounded maps
\b*
(b,\sigma,\sigma^0):\R_+\times\R^d\times\R^d \times\Pc_2(\R^d)\times A
&\longrightarrow&
\R^d\times\Mc_{d,d}(\R)\times\Mc_{d,d_0}(\R),
\e*
and
\b*
(k, \gamma, \gamma^0):\R_+\times\R^d 
&\longrightarrow&
\R^d\times\Mc_{d,d}(\R)\times\Mc_{d,d_0}(\R).
\e*
For $t \geq 0$ and $m \in \mathcal{P}_2(\mathbb{R}^d)$, we denote by $\mathcal{P}(t, y, m)$ the collection of all probability measures $\mathbb{P}$ on $(\Omega,\Fc)$ satisfying: 
\begin{itemize} 
\item[(i)] $W^0$ is a $\mathbb{P}$-Brownian Motion;
\item[(ii)] The process $Y$ is defined by
\b*
Y_t=y,
&\mbox{and}&
\dd Y_s = k(s, Y_s) \dd s + \gamma (s,  Y_s) \dd B^\P_s + \gamma^0(s, Y_s) \dd W^0_s,~s\ge t,~\P-\mbox{a.s}
\e*
for some $\P$-Brownian motion $B^\P$;
\item[(iii)] the conditional marginal law of $X_t$ given $W^0$ is $\mathbb{P}_{X_t}^{W^0}:=\mathbb{P} \circ (X_t | W^0)^{-1} = m$, and there exists a control process $\alpha\in\Ac$ such that for $s \geq t$:
\begin{align*}
\dd X_s 
= b(s, X_s, Y_s, \mathbb{P}_{X_s}^{W^0}, \alpha_s) \dd s 
&+ \sigma(s, X_s, Y_s, \mathbb{P}_{X_s}^{W^0}, \alpha_s) \dd W_s^\mathbb{P} 
\\
&+ \sigma^0(s, X_s, Y_s, \mathbb{P}_{X_s}^{W^0}, \alpha_s) \dd W_s^0,
\end{align*}
$\P-$a.s. for some $\mathbb{P}$-Brownian motion $W^\mathbb{P}$. Here, $\mathbb{P}_{X_s}^{W^0}$ is the conditional law of $X_s$ under $\mathbb{P}$ given $\lbrace W_r^0, r \geq 0 \rbrace$. 
\end{itemize}	
		
We define the objective function as 
\b*
J(t, y, m, \mathbb{P}) 
:= 
\mathbb{E}_t \left[ \int_t^T f^{\alpha_s}(Y_s, \mathbb{P}_{X_s}^{W^0})\dd s 
                             + g(Y_T, \mathbb{P}_{X_T}^{W^0})
                      \right],
&\mbox{for all}&
\P\in\Pc(t, y, m),
\e*
with running reward map $f : A \times \R^d \times \mathcal{P}_2(\mathbb{R}^d) \rightarrow \mathbb{R}$, and final reward $g : \mathcal{P}_2(\mathbb{R}^d) \rightarrow \R^d \times \mathbb{R}$. The dynamic version of the control problem is defined by:
		\begin{equation*}
			V(t, y, m) := \sup_{\mathbb{P} \in \mathcal{P}(t,y, m)} J(t, y, m, \mathbb{P}).
		\end{equation*}
We start from the Dynamic Programming Principle (DPP) which holds under fairly general assumptions, see e.g. Djete, Possama\"{\i} \& Tan \cite{djete2020}:
\begin{equation*}
V(t, y, m) 
= 
\sup_{\mathbb{P} \in \mathcal{P}(t,y,m)} 
\mathbb{E}_t \left[ \int_t^{\theta^\mathbb{P}} f^{\alpha_s}(X_s, Y_s, \mathbb{P}_{X_s}^{W^0}) \dd s + V(\theta^\mathbb{P}, Y_{\theta^\P}, \mathbb{P}_{X_{\theta^\mathbb{P}}}^{W^0})       \right],
		\end{equation*}
for any family $\lbrace \theta^\mathbb{P} \rbrace_{\mathbb{P} \in \mathcal{P}(t,m)}$ of $[t, T]$-valued stopping times.
Our objective is to derive the HJB equation as the infinitesimal counterpart of the last dynamic programming equation under smoothness conditions on the value function. It is well known that such strong smoothness conditions can (and should) be avoided by introducing an appropriate notion of viscosity solution. This important point is beyond the scope of the current illustrative section and is left for future work.

		\begin{proposition}
			Suppose that $\partial_t V, \partial_y V, \partial_y^2 V, \delta_m V, \partial_x \delta_m V, \partial_x^2 \delta_m V, \delta_m^2 V, \partial_x \partial_{\hat{x}} \delta_m^2 V, \partial_x \delta_m \partial_y V$ exist and are jointly continuous, with respect to both $\mathcal{W}_2$ and $\mathcal{W_1}$ for the measure valued argument.
Then $V$ satisfies the following Hamilton-Jacobi-Bellman (HJB) equation:
\b* 
0
&=&- \partial_t V - k\!\cdot\!\partial_y V
       -\frac12(\gamma\gamma^\intercal+\gamma^0\gamma^{0^\intercal})\!:\!\partial_{yy} V
\\
&& \hspace{-10mm}
- \!\!\!\sup_{a \in \L^0(A)} \!
  \Big\{ \!\!\int\!\! \Big(f^a\!+\! b^a\!\!\cdot\!\partial_x \delta_m V 
                                   \!+\!\frac12\big(\sigma^a\sigma^{a^\intercal}
                                                       \!\!\!+\!\sigma^{0,a}\!(\sigma^{0,a})^\intercal
                                                \!\big)\!:\!\partial_{xx}\delta_m V
                                   \!\!\!+\!\sigma^{0,a} (\gamma^0)^\intercal
                                                      \!\!:\!\partial_x \delta_m \partial_y V 
                                            \Big)(.,x)m(\dd x)
\\
&&\hspace{14mm}
+ \frac{1}{2} \int \!\!\! \int \sigma^{0,a}(.,x) \sigma^{0,a}(.,\hat x)\!:\!\partial_{x \hat{x}}^2 \delta_m^2 V(., x, \hat{x})  m(\dd x)m(\dd \hat{x})
\Big\}, \vspace{2mm}
\\
&&\hspace{-10mm}
V\big|_{t=T}
= g.
\e*
where we denoted $\varphi^a(.,x):=\varphi\big(t,y,m,x,a(x)\big)$ for all function $\varphi$.
\end{proposition}
		
		\begin{proof}
			We write the proof for $d = d_0 = 1$.
			\begin{itemize}
				\item \textbf{Supersolution property:} For $t \in [0, T]$,  let $\alpha$ be a constant control process, such that $\alpha_t = a \in A$ for all $t \in [0,T]$. By the Dynamic Programming Principle, for any stopping time $\theta$, 
				\begin{equation*}
					V(t, y, m) \geq \mathbb{E}_t \left[ \int_t^\theta f^{\alpha_s}(Y_s, \mu_s) \dd s + V(\theta, Y_\theta, \mu_\theta) \right],
				\end{equation*}
				that is,
				\begin{equation} \label{HJB1}
					\mathbb{E}_t \left[ \int_t^\theta f^{\alpha_s}(Y_s, \mu_s) \dd s + V(\theta, Y_\theta, \mu_\theta) - V(t, y, m)  \right] \leq 0.
				\end{equation}
				By It\=o-Wentzell's formula,
\b*
V(\theta, Y_\theta, \mu_\theta) - V(t, y, m) 
&=& 
\int_t^\theta \mathcal{L}^{\alpha} V(s, Y_s, \mu_s) \dd s + \int_t^\theta \partial_y V(s, Y_s, \mu_s) \gamma_s \dd B_s^{\mathbb{P}}
\\
&&\hspace{-20mm}
+ \int_t^\theta \left(\mathbb{E}^0 \left[  \partial_x \delta_m V(s, Y_s, \mu_s, X_s) \sigma_s^0 \right] + \partial_y V(s, Y_s, \mu_s) \gamma_s^0 \right) \dd W_s^0,
\e*
where 
\b*
\mathcal{L}^{a} V(s, y, m) 
&:=& 
\partial_t V(s, y, m) + \partial_y V(s, y, m)k_s 
					+\dfrac{1}{2} 
					 \partial_y^2 ((\gamma_s^a)^2 + (\gamma_s^{0,a})^2)
					  V(s, y, m) 
\\
&&
\hspace{-25mm}+  \int \partial_x \delta_m V(s, y, m, x) b_s^a m(\dd x) + \dfrac{1}{2} \int \partial_x^2 \delta_m V(s, y, m, x) ((\sigma_s^a)^2 + (\sigma_s^{0,a})^2) m (\dd x) 
\\
&&
\hspace{-25mm}+ \int \partial_x \delta_m \partial_y V(s, y, m, x) \sigma_s^{0,a} \gamma_s^{0,a} m(\dd x) 
\\
&&
\hspace{-25mm}
+ \dfrac{1}{2} \int\!\!\!\! \int \partial_{x \hat{x}}^2 \delta_m^2 V(s, y, m, x, \hat{x}) \sigma_s^{0,a} \hat{\sigma}_s^{0,a} m (\dd x) m (\dd \hat{x}).
\e*
Let us now choose $\theta = \theta_h := \inf \lbrace s > t, \mathcal{W}_2(\mu_s, \mu_t) \geq 1 \text{ or } |Y_s - Y_t| \geq 1\rbrace \wedge (t+h)$, for $h > 0$. Since $X, Y, \sigma^0, \gamma, \gamma^0, \mu, \partial_y V$ and $\partial_x \delta_m V$ are continuous, $\mathbb{E}^0 \left[  \partial_x \delta_m V(s, Y_s, \mu_s, X_s) \sigma_s^0 \right]$, $\partial_y V(s, Y_s, \mu_s) \gamma_s^0$ and $\partial_y V(s, Y_s, \mu_s) \gamma_s$ are bounded over $[t, \theta_h]$ and therefore the conditional expectation of the stochastic integral terms in the last expression is zero.
				Then, dividing \eqref{HJB1} by $h > 0$:
				\begin{equation*}
					\mathbb{E}_t \left[ \dfrac{1}{h} \int_t^{\theta_h} (f^{\alpha_s}(Y_s,  \mu_s) + \mathcal{L}^{\alpha_s} V(s, Y_s, \mu_s)) \dd s \right] \leq 0.
				\end{equation*}
				But a.s., for $h$ small enough $\theta_h = t+h$ and
				\begin{equation*}
					\dfrac{1}{\theta_h - t} \int_t^{\theta_h} (f^{\alpha_s}(Y_s, \mu_s) + \mathcal{L}^{\alpha_s} V(s, Y_s, \mu_s)) \dd s \longrightarrow f^{a}(y, m) + \mathcal{L}^a V(t, y, m), \text{ a.s.}
				\end{equation*}
				as $\alpha$ is constant. By dominated convergence for the expectation, we then have that
				\begin{equation*}
					f^{a}(y, m) + \mathcal{L}^{a} V(t, y, m) \leq 0.
				\end{equation*}
				Since $a \in A$ is arbitrary, we can conclude that
				\begin{equation*}
					\sup_{a \in \L^0(A)} \mathcal{L}^a V (t, y, m) + f^a(y, m) \leq 0.
				\end{equation*}
				
				\item  \textbf{Subsolution property.
Let $\tau^\mathbb{P}_h:=(t+h)\wedge\tau^\mathbb{P}$ with $\tau^\mathbb{P}:=\inf\{s>t:\mathcal{W}_2(\mu_s,m) + | Y_s - y | \ge\delta\}$, for some $\delta>0$, and $\mu_s := \mathbb{P}^0_{X_{s}}$.
By the dynamic programming principle, we have: 
$$
V(t, y, m) \le \sup_{\mathbb{P} \in \mathcal{P}(t,y,m)} \mathbb{E}^\mathbb{P} \Big[ \int_t^{\tau^\mathbb{P}_h} f^{\alpha_s}\big(Y_s, \mu_s\big)ds + V\big(\tau^\mathbb{P}_h,\mu_{\tau^\mathbb{P}_h}\big)\Big],
$$
from which we deduce, by Itô-Wentzell formula:
\begin{align*}
\inf_{\mathbb{P} \in \mathcal{P}(t, y, m)} \int_t^{\tau^\mathbb{P}_h} 
\mathbb{E}^{ \mathbb{P}}\Big[ &(-\! \mathcal{L}^{\alpha_s} \delta_m V  
                               \!-\! f^{\alpha_s})(s,  Y_s, \mu_s,X_s) \Big]ds \le 0,
\end{align*}
For $\delta > 0$, introduce now $B_2(y,\mu,\delta) := \{ (y',\mu') \in \mathbb{R} \times \mathcal{P}_2(\mathbb{R}^d) : |y' - y| + \mathcal{W}_2(\mu, \mu') \le \delta \}$ and $Q_2^{h,\delta}(t,y,\mu) := [t,t+h] \times B_2(\mu,y,\delta)$. Then it follows from the local boundedness of the derivatives of $V$ that: 
\begin{eqnarray*}
\inf_{(s,y',\mu') \in Q_2^{h,\delta}(t,y,\mu)} \mathcal{H}(s,y',\mu')
\le
\frac{C_\delta}{h}\; \sup_{\mathbb{P} \in \mathcal{P}(t,\mu)} \int_t^{t+h} \mathbb{P}\big[(Y_s, \mu_s) \notin B_2(\mu,y, \delta)\big]ds,
\end{eqnarray*}
for some constant $C_\delta> 0$, where
$$
\mathcal{H}(s,y,\mu')
:=
-\int_{\mathbb{R}} \sup_{a \in A} \Big\{ \mathcal{L}^{a} V +f^a 
                                                                                           \Big\}(s,y,\mu',x) \mu'(dx).
$$
As the control process $\alpha$ take values in a the compact subset $A$, it follows from the continuity of the coefficients of the controlled SDE that:
\begin{align*}
\mathbb{P}\big[ (Y_s, m_s) \notin B_2(\mu, \delta)\big] \le& \mathbb{P}\big[ | Y_s - y | \ge \delta \big] + \mathbb{P}\big[ \mathcal{W}_2(m_s, \mu) \ge \delta \big] \\
\le& \frac{\mathbb{E}^\mathbb{P}\big[ |Y_s - y|^2 + | X_s - X_t |^2 \big]}{\delta^2} \le C' \;\frac{s-t}{\delta^2},
\end{align*}
for some constant $C'>0$. Together with the continuity of $\mathcal{H}$, due to the compactness of $A$, this implies that:
\begin{eqnarray*}
\inf_{(y',\mu') \in B_2(y, \mu, \delta)} \mathcal{H}(t,y',\mu')
\;=\;
\lim_{h\searrow 0}
\inf_{(s,y',\mu') \in Q_2^{h,\delta}(t,y,\mu)}  \mathcal{H}(s,y', \mu')
&\le& 
\lim_{h\searrow 0}\frac{C'C_\delta h}{2\delta^2}
\;=\;
0.
\end{eqnarray*}
To conclude, we use the fact that, since all derivatives of $V$ involved in $\mathcal{H}$ are in continuous both on $\mathcal{W}_1$ and $\mathcal{W}_2$, it follows that $\mathcal{H}$ is uniformly continuous on $B_2(\mu, 1)$ for the $\mathcal{W}_1$ distance. Since $\mathcal{W}_1 \le \mathcal{W}_2$, we obtain by sending $\delta \searrow 0$ that $\mathcal{H}(t,y,\mu)\le 0$, which is the required subsolution property. }
\qed

		\end{itemize}		
		\end{proof}	
		
	
\section{Appendix: Proof of Lemma \ref{lemma}}
\label{sec:App}

For simplicity, we only report the proof for $d=1$, as the extension to arbitrary dimension does not raise any difficulty. Let us first notice that
		\begin{equation*}
			\sum_{i=1}^{p_n} H^n_{t_{i-1}^n} (\Dpin \langle M \rangle_{t_i^n}) \longrightarrow \int_0^T H_s \dd \langle M \rangle_s \text{ in } \mathbb{L}^1.
		\end{equation*}
		Now, with transparent notations, it is obvious that
		\begin{align*}
			\Big\| \sum_{i=1}^{p_n}  H_{t_{i-1}^n} \left((\Delta^{\pi_n} X_{t_i^n})^2 -\Dpin \langle M \rangle_{t_i^n}\right) \Big\|_1  &\leq \Big\| \sum_{i=1}^{p_n}  H_{t_{i-1}^n} (\Delta^{\pi_n} A_{t_i^n})^2 \Big\|_1  \\
			&+ \Big\| \sum_{i=1}^{p_n}  H_{t_{i-1}^n} \left((\Delta^{\pi_n} M_{t_i^n})^2 -\Dpin \langle M \rangle_{t_i^n}\right) \Big\|_1 \\
			&+ 2 \Big\| \sum_{i=1}^{p_n}  H_{t_{i-1}^n} \Delta^{\pi_n} A_{t_i^n} \Delta^{\pi_n} M_{t_i^n} \Big\|_1.
		\end{align*}
		For the first term on the right-hand side: writing $||H||_\infty$ for a uniform, deterministic bound on $|H^n|, n \geq 1$, 
		\begin{equation*}
			\Big| \sum_{i=1}^{p_n}  H^n_{t_{i-1}^n} (\Dpin A_{t_i^n} )^2 \Big| \leq ||H||_\infty \sum_{i=1}^{p_n} (\Dpin A_{t_i^n} )^2  =  ||H||_\infty {\rm QV}_{\pi^n} (A)
		\end{equation*}
		where ${\rm QV}_{\pi^n} (A)$ is the quadratic variation of the finite-variation process $A$ along the partition ${\pi}^n$.
		As ${\rm QV}_{\pi^n}(A) \longrightarrow 0$ a.s. and ${\rm QV}_{\pi^n} (A) \leq |A|_{\rm TV}^2$ which is in $\mathbb{L}^1$, it follows from the dominated convergence theorem that
		\begin{equation*}
			\Big\| \sum_{i=1}^{p_n}  H^n_{t_{i-1}^n} (\Dpin A_{t_i^n} )^2 \Big\|_1 \longrightarrow 0.
		\end{equation*}
		For the middle term, we introduce the martingale defined by $R_t^{t_{i-1}^n} :=  (\Dpin M_t)^2 -\Dpin \langle M \rangle_{t}$ for $t \geq t_{i-1}^n$ and we now show that
		\begin{equation*}
			\Big\| \sum_{i=1}^{p_n - 1} H^n_{t_{i-1}^n} R_{t_i^n}^{t_{i-1}^n}  \Big\|_2^2 = \Big\| \sum_{i=1}^{p_n}  H^n_{t_{i-1}^n} \left((\Dpin M_{t_i^n})^2 -\Dpin \langle M \rangle_{t_i^n}\right) \Big\|_2^2 \longrightarrow 0.
		\end{equation*}
		To see this, we directly compute that
		\begin{align*}
			\Big\| \sum_{i=1}^{p_n - 1} H^n_{t_{i-1}^n} R_{t_i^n}^{t_{i-1}^n}  \Big\|_2^2 &= \mathbb{E} \left[ \sum_{i=1}^{p_n} {H^n_{t_{i-1}^n}}^2 (R_{t_i^n}^{t_{i-1}^n})^2 \right] + 2 \mathbb{E} \left[ \sum_{0 \leq i < j \leq n-1} H^n_{t_{i-1}^n} H^n_{t_j^n} R^{t_{i-1}^n}_{t_i^n} R^{t_j^n}_{t_{j+1}^n}\right] \\
			&\hspace{-12mm}= \mathbb{E} \left[ \sum_{i=1}^{p_n} {H^n_{t_{i-1}^n}}^2 (R_{t_i^n}^{t_{i-1}^n})^2 \right] + 2 \sum_{0 \leq i < j \leq n-1} \mathbb{E} \left[ H^n_{t_{i-1}^n} H^n_{{t_j^n}^n} R^{t_{i-1}^n}_{t_i^n} \mathbb{E}\left[ R^{t_j^n}_{t_{j+1}^n} | \mathcal{F}_{t_j^n} \right] \right].
		\end{align*}
		As $\mathbb{E}\left[ R^{t_j^n}_{t_{j+1}^n} | \mathcal{F}_{t_j^n} \right] = 0$, this implies that
		\begin{equation*}
			\Big\| \sum_{i=1}^{p_n}  H^n_{t_{i-1}^n} ((\Dpin M_{t_i^n} )^2 -\Dpin \langle M \rangle_{t_i^n}) \Big\|_2^2 \leq ||H||_\infty^2 \sum_{i=1}^{p_n} \mathbb{E} \left[ (R_{t_i^n}^{t_{i-1}^n})^2 \right].
		\end{equation*}
		We next estimate that
		\begin{align*}
			\mathbb{E} \left[ (R_{t_i^n}^{t_{i-1}^n})^2 \right]	&=\mathbb{E} \left[ (\Dpin M_{t_i^n} )^4 - 2 (\Dpin M_{t_i^n})^2 \Dpin \langle M \rangle_{t_i^n} + (\Dpin \langle M \rangle_{t_i^n})^2 \right] \\
			&\leq \mathbb{E} \left[ (\Dpin M_{t_i^n} )^4 \right] + \mathbb{E} \left[ (\Dpin \langle M \rangle_{t_i^n})^2 \right] \\
			&\leq (1 + C_4) \mathbb{E} \left[(\Dpin \langle M \rangle_{t_i^n})^2 \right]
		\end{align*}
		for some constant $C_4$ induced by the BDG inequality for the order $p=4$. Therefore,
		\begin{equation*}
			\Big\| \sum_{i=1}^{p_n}  H^n_{t_{i-1}^n} \left((\Dpin M_{t_i^n})^2  - \Dpin \langle M \rangle_{t_i^n}\right) \Big\|_2^2  \leq  (1 + C_4) ||H||_\infty^2 \mathbb{E} \left[ {\rm QV}_\pi^n(\langle M \rangle) \right]
		\end{equation*}
		where ${\rm QV}_\pi^n(\langle M \rangle) = \sum_{i=1}^{p_n} (\langle M \rangle_{t_i^n} - \langle M \rangle_{t_{i-1}^n})^2$. Since $\langle M \rangle$ is a finite-variation process, ${\rm QV}_\pi^n (\langle M \rangle) \longrightarrow 0$ almost surely as $n \rightarrow \infty$, and since ${\rm QV}_\pi^n (\langle M \rangle) \leq \langle M \rangle_T^2 \in \mathbb{L}^1$  by Condition {\rm(IW3)}, we conclude by dominated convergence.
		
		Finally, for the last term, note that the previous calculations, for $H = 1$, show that $\sum_i (\Dpin M_{t_i^n} )^2 \longrightarrow \langle M \rangle_T$ in $\mathbb{L}^2$. Then, by applying the Cauchy-Schwarz inequality twice:
		\begin{align*}
			\mathbb{E} \left[ \Big| \sum_{i=1}^{p_n} H^n_{t_{i-1}^n} (\Dpin A_{t_i^n} )(\Dpin M_{t_i^n} ) \Big| \right]	&\leq ||H||_{\infty} \mathbb{E} \left[\sqrt{\sum_{i=1}^{p_n}  (\Dpin A_{t_i^n} )^2} \sqrt{\sum_{i=1}^{p_n} (\Dpin M_{t_i^n} )^2}  \right] \\
			&\leq ||H||_{\infty} \Big\| \sum_{i=1}^{p_n} (\Dpin A_{t_{i}^n})^2 \Big\|_{2}  \Big\| \sum_{i=1}^{p_n} (\Dpin M_{t_{i}^n})^2 \Big\|_2 \longrightarrow 0,
		\end{align*}
		since $\mathbb{E} \left[ \sum_{i=1}^{p_n} (\Dpin A_{t_i^n} )^2 \right] \rightarrow 0$ by dominated convergence (see first term calculations).

%
%

\end{document}